\documentclass[12pt]{amsart}
\calclayout
\usepackage{amssymb}
\usepackage{amsmath}
\usepackage{graphics}
\usepackage{epsfig, color}
\newcommand\R{\mathbb R}

\newcommand\C{\mathcal C}
\newcommand\K{\mathcal K}
\renewcommand\O{\mathcal O}
\newcommand\T{\mathcal T}
\newcommand\E{\mathcal E}
\newcommand\M{\mathcal M}

\newcommand\N{\mathcal N}
\newcommand\V{\mathcal V}

\renewcommand\H{\mathcal H}

\newcommand\eps{\operatorname{\epsilon}}

\newcommand\x{\times}
\renewcommand\t{\tilde}
\newcommand\lbra{[\![}
\newcommand\rbra{]\!]}
\newcommand\lbrac{\,[\!\!\!\{}
\newcommand\rbrac{\}\!\!\!]\,}

\renewcommand\ll{|\kern-2pt|\kern-2pt|}

\newcommand\e{^\epsilon}

\numberwithin{equation}{section}
\theoremstyle{plain}
\newtheorem{thm}{Theorem}

\newtheorem{lem}[thm]{Lemma}

\numberwithin{thm}{section}

\theoremstyle{remark}

\def\underput#1#2#3{
\mathchoice
{\vtop{\ialign{##\crcr\hfil$#2$\vrule width0pt height0pt depth#3\hfil\crcr
\noalign{\nointerlineskip}\hfil$\scriptstyle#1$\hfil\crcr}}}
{\vtop{\ialign{##\crcr\hfil$#2$\vrule width0pt height0pt depth#3\hfil\crcr
\noalign{\nointerlineskip}\hfil$\scriptstyle#1$\hfil\crcr}}}
{\vtop{\ialign{##\crcr\hfil$\scriptstyle#2$\vrule width0pt height0pt
depth#3\hfil\crcr
\noalign{\nointerlineskip}\hfil$\scriptscriptstyle#1$\hfil\crcr}}}
{\vtop{\ialign{##\crcr\hfil$\scriptscriptstyle#2$\vrule width0pt height0pt
depth#3\hfil\crcr
\noalign{\nointerlineskip}\hfil$\scriptscriptstyle#1$\hfil\crcr}}}}

\def\stack#1#2#3{\rlap{#1}\lower#3\hbox{#2}}

\def\twiddlespace{1.2truept}

\def\dtwiddle{\displaystyle\sim}
\def\ttwiddle{\textstyle\sim}
\def\stwiddle{\scriptstyle\sim}
\def\sstwiddle{\scriptscriptstyle\sim}

\def\doubledtwiddle{\stack{$\dtwiddle$}{$\dtwiddle$}{\twiddlespace}}
\def\doublettwiddle{\stack{$\ttwiddle$}{$\ttwiddle$}{\twiddlespace}}
\def\doublestwiddle{\stack{$\stwiddle$}{$\stwiddle$}{\twiddlespace}}
\def\doublesstwiddle{\stack{$\sstwiddle$}{$\sstwiddle$}{\twiddlespace}}

\def\tripledtwiddle{\stack{$\dtwiddle$}{$\doubledtwiddle$}{\twiddlespace}}
\def\triplettwiddle{\stack{$\ttwiddle$}{$\doublettwiddle$}{\twiddlespace}}
\def\triplestwiddle{\stack{$\stwiddle$}{$\doublestwiddle$}{\twiddlespace}}
\def\triplesstwiddle{\stack{$\sstwiddle$}{$\doublesstwiddle$}{\twiddlespace}}

\def\quadrupledtwiddle{\stack{$\dtwiddle$}{$\tripledtwiddle$}{\twiddlespace}}
\def\quadruplettwiddle{\stack{$\ttwiddle$}{$\triplettwiddle$}{\twiddlespace}}
\def\quadruplestwiddle{\stack{$\stwiddle$}{$\triplestwiddle$}{\twiddlespace}}
\def\quadruplesstwiddle{\stack{$\sstwiddle$}{$\triplesstwiddle$}{\twiddlespace}}\def\quadru
pletwiddle{{\mathchoice{\quadrupledtwiddle}{\quadruplettwiddle}%

{\quadruplestwiddle}{\quadruplesstwiddle}}}

\def\strikedist{3pt}

\def\dstrike{\vrule width7pt height0pt depth.4pt}
\def\tstrike{\vrule width7pt height0pt depth.4pt}
\def\sstrike{\hbox{\vrule width5pt height0pt depth.4pt}}
\def\ssstrike{\hbox{\vrule width3pt height0pt depth.4pt}}
\def\strike{{\mathchoice{\dstrike}{\tstrike}{\sstrike}{\ssstrike}}}

\def\ub#1{\underput\strike{#1}{\strikedist}}

\catcode`\!=11
\newcommand\!a{{\boldsymbol a}}
\newcommand\!b{{\boldsymbol b}}
\newcommand\!c{{\boldsymbol c}}
\newcommand\!d{{\boldsymbol d}}
\newcommand\!e{{\boldsymbol e}}
\newcommand\!f{{\boldsymbol f}}
\newcommand\!g{{\boldsymbol g}}
\newcommand\!h{{\boldsymbol h}}
\newcommand\!i{{\boldsymbol i}}
\newcommand\!j{{\boldsymbol j}}
\newcommand\!k{{\boldsymbol k}}
\newcommand\!l{{\boldsymbol l}}
\newcommand\!m{{\boldsymbol m}}
\newcommand\!n{{\boldsymbol n}}
\newcommand\!o{{\boldsymbol o}}
\newcommand\!p{{\boldsymbol p}}
\newcommand\!q{{\boldsymbol q}}
\newcommand\!r{{\boldsymbol r}}
\newcommand\!s{{\boldsymbol s}}
\newcommand\!t{{\boldsymbol t}}
\newcommand\!u{{\boldsymbol u}}
\newcommand\!v{{\boldsymbol v}}
\newcommand\!w{{\boldsymbol w}}
\newcommand\!x{{\boldsymbol x}}
\newcommand\!y{{\boldsymbol y}}
\newcommand\!z{{\boldsymbol z}}
\newcommand\!A{{\boldsymbol A}}
\newcommand\!B{{\boldsymbol B}}
\newcommand\!C{{\boldsymbol C}}
\newcommand\!D{{\boldsymbol D}}
\newcommand\!E{{\boldsymbol E}}
\newcommand\!F{{\boldsymbol F}}
\newcommand\!G{{\boldsymbol G}}
\newcommand\!H{{\boldsymbol H}}
\newcommand\!I{{\boldsymbol I}}
\newcommand\!J{{\boldsymbol J}}
\newcommand\!K{{\boldsymbol K}}
\newcommand\!L{{\boldsymbol L}}
\newcommand\!M{{\boldsymbol M}}
\newcommand\!N{{\boldsymbol N}}
\newcommand\!O{{\boldsymbol O}}
\newcommand\!P{{\boldsymbol P}}
\newcommand\!Q{{\boldsymbol Q}}
\newcommand\!R{{\boldsymbol R}}
\newcommand\!S{{\boldsymbol S}}
\newcommand\!T{{\boldsymbol T}}
\newcommand\!U{{\boldsymbol U}}
\newcommand\!V{{\boldsymbol V}}
\newcommand\!W{{\boldsymbol W}}
\newcommand\!X{{\boldsymbol X}}
\newcommand\!Y{{\boldsymbol Y}}
\newcommand\!Z{{\boldsymbol Z}}
\newcommand\!alpha{{\boldsymbol\alpha}}
\newcommand\!beta{{\boldsymbol\beta}}
\newcommand\!gamma{{\boldsymbol\gamma}}
\newcommand\!delta{{\boldsymbol\delta}}
\newcommand\!epsilon{{\boldsymbol\epsilon}}
\newcommand\!zeta{{\boldsymbol\zeta}}
\newcommand\!eta{{\boldsymbol\eta}}
\newcommand\!theta{{\boldsymbol\theta}}
\newcommand\!iota{{\boldsymbol\iota}}
\newcommand\!kappa{{\boldsymbol\kappa}}
\newcommand\!lambda{{\boldsymbol\lambda}}
\newcommand\!mu{{\boldsymbol\mu}}
\newcommand\!nu{{\boldsymbol\nu}}
\newcommand\!xi{{\boldsymbol\xi}}
\newcommand\!pi{{\boldsymbol\pi}}
\newcommand\!rho{{\boldsymbol\rho}}
\newcommand\!sigma{{\boldsymbol\sigma}}
\newcommand\!tau{{\boldsymbol\tau}}
\newcommand\!upsilon{{\boldsymbol\upsilon}}
\newcommand\!phi{{\boldsymbol\phi}}
\newcommand\!chi{{\boldsymbol\chi}}
\newcommand\!psi{{\boldsymbol\psi}}
\newcommand\!omega{{\boldsymbol\omega}}
\newcommand\!varepsilon{{\boldsymbol\varepsilon}}
\newcommand\!vartheta{{\boldsymbol\vartheta}}
\newcommand\!varpi{{\boldsymbol\varpi}}
\newcommand\!varrho{{\boldsymbol\varrho}}
\newcommand\!varsigma{{\boldsymbol\varsigma}}
\newcommand\!varphi{{\boldsymbol\varphi}}
\newcommand\!Gamma{{\boldsymbol\Gamma}}
\newcommand\!Delta{{\boldsymbol\Delta}}
\newcommand\!Theta{{\boldsymbol\Theta}}
\newcommand\!Lambda{{\boldsymbol\Lambda}}
\newcommand\!Xi{{\boldsymbol\Xi}}
\newcommand\!Pi{{\boldsymbol\Pi}}
\newcommand\!Sigma{{\boldsymbol\Omega\eigma}}
\newcommand\!Upsilon{{\boldsymbol\Upsilon}}
\newcommand\!Phi{{\boldsymbol\Phi}}
\newcommand\!Psi{{\boldsymbol\Psi}}
\newcommand\!Omega{{\boldsymbol\Omega}}

\begin{document}

\title [Finite element for Koiter shell]
{Analysis of a discontinuous Galerkin method\\ for Koiter shell}


\author{Sheng Zhang}

\thanks{Department of Mathematics, Wayne State University, Detroit, MI 48202}

\begin{abstract}
We present an analysis for a mixed finite element method for the bending problem of Koiter shell.
We derive an error estimate showing that when the geometrical coefficients of the shell mid-surface satisfy certain conditions
the finite element method has the optimal order of accuracy, which is uniform with respect to the shell
thickness. Generally, the error estimate shows how the accuracy  is affected by the shell geometry and thickness.
It suggests that  to achieve optimal rate of convergence, the triangulation should be
properly refined in regions where the shell geometry changes dramatically.
The analysis is carried out for a balanced method in which the normal component of displacement is approximated by discontinuous
piecewise cubic polynomials, while the tangential components are approximated by
discontinuous piecewise quadratic polynomials, with some enrichment on elements that
have edges on the free boundary.
Components of the membrane stress 
are approximated by continuous piecewise linear functions.

\vspace{12pt}

\end{abstract}
\maketitle


\section{Introduction}

We analyze the accuracy of a mixed finite element method for the Koiter shell model, in which the shell displacement variables are
approximated by discontinuous piecewise polynomials, while the membrane stress components are approximated by
continuous piecewise polynomials.
This is a discontinuous Galerkin (DG) method in terms of the primary variables of Koiter shell.
The finite elements for various variables  form a balanced combination in the sense that
except for some minor tangential displacement enrichments required by stability on the free edge of the shell,
every degree of freedom contributes to the accuracy of the finite element solution.
DG method provides a more general approach
and offers  more flexibilities in choosing finite element spaces and degree of freedoms.
It is believed to have a potential to help resolve some difficult problems in numerical computation of elastic shells \cite{AF-RM-DG, A-RM-DG}.
In this paper, we show that DG method indeed has advantages in reducing the troublesome
membrane locking in computation of shell bending problems.
We prove an error estimate showing that when the geometry of a shell satisfies certain conditions the method yields
a finite element solution that has the
optimal order of accuracy that could be achieved by the best approximation from the finite element functions.
Thus it is free of membrane locking.
When such condition is not satisfied, the estimate shows how the accuracy is affected
by the geometrical coefficients and how to adjust the finite element mesh to accommodate the curved shell mid-surface such that
the finite element solution
achieves the optimal order of accuracy.
Particularly, the estimate suggests that some refinements for the finite element mesh should be done
where a shell changes geometry abruptly.

We consider a thin shell of thickness $2\eps$.
Its middle surface $\t\Omega\subset\R^3$ is the image of a two-dimensional coordinate domain $\Omega\subset\R^2$ through
the parameterization mapping $\!phi:\Omega\to\t\Omega$. This mapping furnishes the curvilinear coordinates on the surface $\t\Omega$.
Subject to loading forces and boundary conditions, the shell would be deformed to a stressed state.
The Koiter shell model uses displacement of the shell mid-surface as the primary variables.
The tangential displacement is represented by its covariant components $u\e_{\alpha}$ ($\alpha\in\{1,2\}$), and normal displacement is a scalar $w\e$.
The superscript $\eps$ indicates dependence on the shell thickness.
To deal with membrane locking, we also introduce
the symmetric membrane stress tensor scaled by multiplying the factor $\eps^{-2}$  as an independent variable,
which is given in terms its contravariant components $\M^{\eps\alpha\beta}$ ($\alpha, \beta\in\{1,2\}$).
All the six functions are two-variable functions defined on $\Omega$.
For a bending dominated shell problem, under a suitable scaling on the loading force, these functions converge to finite limits when $\eps\to 0$.
This justifies our choice of approximating them as independent variables.
For a curved shell deformation to be bending dominated,  the shell needs to have a portion of its boundary free, 
or subject to
force conditions. It is known that a totally clamped or simply supported  elliptic, parabolic, or hyperbolic
shell does not allow bending dominated behavior \cite{CiarletIII}.
We assume the shell boundary is divided into three parts, on which the shell is clamped, simply supported, and free of displacement constraint,
respectively, and the free part is not empty.

We assume that the coordinate domain $\Omega$ is a polygon. On $\Omega$, we introduce a triangulation $\T_h$
that is shape regular but not necessarily quasi-uniform. The shape regularity of a triangle is defined as the ratio of the diameter
of its smallest circumscribed circle and the diameter of its largest inscribed circle. The shape regularity of a triangulation is the maximum of
shape regularities of all its triangular elements. When we say a $\T_h$ is shape regular we mean that the shape regularity
of $\T_h$ is bounded by an absolute constant $\K$.  Shape regular meshes allow local refinements, and thus have the potential to
more efficiently resolve the ever increasing singularities in solutions of the shell model.
We use $\T_h$ to
denote the set of all the (open) triangular
elements, and let $\Omega_h=\cup_{\tau\in\T_h}\tau$.
We use $h_\tau$ to denote the diameter of the element $\tau$.
We analyze a particular finite element method, in which
we use totally discontinuous piecewise quadratic polynomials to approximate the
tangential displacement components $u\e_\alpha$, use totally discontinuous
cubic polynomials to approximate the normal deflection $w\e$, and use continuous piecewise linear functions to
approximate the scaled membrane stress tensor components $\M^{\eps\alpha\beta}$.
If an element $\tau$ has one edge that lies on the free boundary of the shell, we need to enrich the space of quadratic
polynomials for the tangential displacement by adding two cubic polynomials. If an element has two edges on the free 
boundary, we need to use the the full cubic polynomials for the tangential displacements. 
The finite element model
yields an approximation $u^h_{\alpha}, w^h, \M^{h \alpha\beta}$,
and we have the error estimate that 
\begin{multline}\label{estimate1}
\|(\!u\e-\!u^h, w\e-w^h)\|
\le C\left[1+\eps^{-1}\max_{\tau\in\T_h; \alpha,\beta,\lambda\in\{1,2\}}\left(h^{3}_\tau|\Gamma^{\lambda}_{\alpha\beta}|_{2,\infty,\tau}+
h^{5}_\tau|b_{\alpha\beta}|_{3,\infty,\tau}\right)
\right]\\
\left[
\sum_{\tau\in\T_h}h^4_{\tau}\left(\sum_{\alpha=1}^2\|u\e_\alpha\|^2_{3,\tau}+\|w\e\|^2_{4,\tau}+
\sum_{\alpha,\beta=1}^2\|\M^{\eps\alpha\beta}\|^2_{2,\tau}
\right)\right]^{1/2}.
\end{multline}
Here, $C$ is a constant that could be dependent on the shape regularity $\K$  of $\T_h$ and the shell mid-surface, but otherwise
it is independent of the finite element mesh, the shell thickness, and the shell model solution.
For a subdomain $\tau\subset\Omega$, we use 
$\|\cdot\|_{k,\tau}$ and $|\cdot|_{k,\tau}$ to denote the norm and semi norm of the Sobolev space $H^k(\tau)$,
and use $\|\cdot\|_{k,\infty,\tau}$ and $|\cdot|_{k,\infty,\tau}$ to denote that of $W^{k,\infty}(\tau)$.
When $\tau=\Omega$, the space $H^k(\Omega)$ will be simply written as $H^k$.
The functions $\Gamma^{\lambda}_{\alpha\beta}$ are
the Christoffel symbols and $b_{\alpha\beta}$ the covariant components of curvature tensor of the parameterized shell middle surface $\t\Omega$. 
These will be called geometrical coefficients of the shell.
The left hand side norm is the  piecewise $H^1$ norm for $u\e_\alpha-u^h_{\alpha}$ and piecewise $H^2$ for $w\e-w^h$,
plus penalties on discontinuity and violation of the essential boundary conditions by the finite element
approximation, see \eqref{Hh-norm} below.
It is noted that we have no estimate in for the error $\M^{\eps\alpha\beta}-\M^{h\alpha\beta}$ in \eqref{estimate1},
while $\M^{\eps\alpha\beta}$ is involved in the
right hand side, which usually has very strong internal and boundary layers. Some weaker estimate for this error will be given below.

The quantity in the first bracket in the right hand side of \eqref{estimate1}
is independent of the shell model solution. It, however, involves the geometrical coefficients,
the triangulation $\T_h$, and the shell thickness $\eps$.
If the curvature tensor components $b_{\alpha\beta}$
are piecewise quadratic functions, and the Christoffel symbols $\Gamma^\lambda_{\alpha\beta}$
are piecewise linear functions, then the quantity is completely
independent of $\eps$.
Generally, $\eps$ has some negative effect. To keep the quantity bounded,
the finite element mesh needs to be relatively fine
where the geometrical coefficients has greater second or third order derivatives.
Where the shell is flat, the thickness $\eps$
does not impose much restriction on the mesh size.
In any case, the quantity in the first bracket is bounded if $h^3=\O(\eps)$, with $h$ being the maximum size of finite elements.
The finite element mesh, the shell shape, and its thickness together should satisfy a condition
such that the quantity in the first bracket is bounded.
The method reduces membrane locking
quite significantly, which otherwise would amplify the error by a factor of the magnitude $\eps^{-1}$.

To assess the accuracy of the finite element solution,
we scale the loading force densities in the shell model by multiplying them with the factor $\eps^2$. (Such scaling will not affect relative errors
of numerical solutions.)
Then we have the limiting behaviors that
when $\eps\to 0$,
$u\e_\alpha\to u^0_\alpha$ in $H^1$, $w\e\to w^0$ in $H^2$, and $\M^{\eps\alpha\beta}$
converges to a limit in a weaker norm.
The shell problem is bending dominated if and only if $(u^0_\alpha, w^0)\neq 0$.
In this case, the smallness of the error in the left hand side of \eqref{estimate1} means small relative error
of the approximation of the primary variables, thus accuracy of the finite element model.
The asymptotic behaviors of $u\e_\alpha$, $w\e$, and $\M^{\eps\alpha\beta}$ , in terms of convergence in strong or weak norms,
mean that they
tend to limiting functions in major  part of the domain, while
may exhibit boundary or internal layers that occur in slimmer and slimmer portions of the domain.
If the finite element functions are capable of
resolving such  singular layers,
the finite element solution would be accurate and free of membrane locking.
It is noted that the quantity in the second bracket in the right hand side of \eqref{estimate1}
is the error estimate of the best approximations of $u\e_\alpha$, $w\e$, and $\M^{\eps\alpha\beta}$
from their finite element functions in the piecewise  $H^1$-norm, piecewise $H^2$-norm, and $L^2$-norm, respectively.

If the limit $(u^0_\alpha, w^0)$ is zero,
the shell deformation is not bending dominated. In this case
we do not have the accuracy of the finite element solution measured in the aforementioned relative error.
In computation, one would obtain finite element solutions that are very small in the norm
in the left hand side of \eqref{estimate1}. The theory implies that
such smallness must not be due to numerical membrane locking.
But rather, it indicates that the shell problem is not bending dominated, and
needs to be treated differently, in which case standard finite element methods could be better.
Whether a shell problem is bending dominated, membrane dominated, or intermediate
is determined by the shell shape, loading forces, and boundary conditions \cite{CiarletIII, Bathe-book}.
Membrane locking is the most critical issue in bending dominated problems \cite{ABrezzi2}.

There is a huge literature on scientific computing and numerical analysis of shell models,  see the books \cite{Bernadou, CiarletIII, Hughes-book, Bathe-book}
for reviews. There are several theories on locking free finite elements that are relevant to this paper.
In \cite{ABrezzi2}, a locking free estimate was established under the assumption that the geometrical coefficients  are piecewise constants.
In \cite{Suri}, similar result was proved for some higher order finite elements under the assumption that
the geometrical coefficients are higher order piecewise polynomials.
These papers did not say how the finite element accuracy would be affected had the assumptions on the geometrical coefficients  not been met.
In \cite{Bramble-Sun2}, a uniform accuracy of a finite element method was proved for Naghdi shell model
under a condition of the form $h^2\le \O(\eps)$, with some bubble functions introduced to enhance the stability.
Our result seems more general than these.
We have made an effort not to assume the finite mesh to be quasi-uniform.  This is important for
the shell model for which layers of singularities are very common in its solution, for which quasi-uniform mesh
is not practical.
The stability achieved in this paper are mainly due to the flexibility of discontinuous approximations. 


The paper is arranged as follows. In Section~\ref{SHELL} we recall the shell model in the standard variational form, 
and write it in a mixed form by introducing the scaled membrane stress as a new variable. 
An asymptotic estimate on the model solution, and an equivalent estimate on the solution of the mixed model
are given in an abstract setting. The latter will also be used in analysis of the finite element model.
In Section~\ref{principle}, we introduce the finite element model that is consistent with the mixed 
form of the Koiter shell. The consistency is verified in the appendix.
In Section~\ref{KornOnShell}, we prove a discrete version of Korn's inequality on shells.
This inequality plays a fundamental role in the error analysis, which is carried out in Section~\ref{ErrorAnalysis}.

For a fixed $\eps$, the shell model solution will be assumed to have the regularity that $u\e_\alpha\in H^3$ and 
$w\e\in H^4$. Of course, when $\eps\to 0$ these functions could  go to 
infinity in these norms.
Throughout the paper, $C$ is a constant that could be dependent on the shell mid-surface and the shape regularity $\K$
of the triangulation  $\T_h$. It is otherwise independent of the triangulation and shell thickness $\eps$. We shall simply say 
that the constant $C$ is independent of $\T_h$ and $\eps$.
For such a constant $C$, we use $A\lesssim B$ to denote $A\le CB$. If $A\lesssim B$ and $A\lesssim B$, we write $A\simeq B$.
Superscripts indicate contravariant components of vectors and tensors, and subscripts indicate covariant components.
Greek sub and super scripts, except $\eps$,  
take their values in $\{1,2\}$. Latin scripts take their values in $\{1,2,3\}$. Summation rules with respect to repeated sub and
super scripts will also be used. A vector with covariant  components $u_\alpha$ or contravariant components $u^\alpha$ 
is represented by the bold face letter $\!u$. 
A tensor with components $\M^{\alpha\beta}$ will be simply called $\M$.


\section{The shell model}
\label{SHELL}
Let $\t\Omega\subset\R^3$
be the middle surface of a shell of thickness $2\eps$.
It is  the image of a domain
$\Omega\subset\R^2$ through a mapping $\!phi$.
The coordinates $x_\alpha\in\Omega$
then furnish the curvilinear coordinates on $\t\Omega$.
We assume that at any point on the surface,
along the coordinate lines,
the two tangential vectors
$\!a_{\alpha}={\partial\!phi}/{\partial x_{\alpha}}$
are linearly independent.
The unit vector
$\!a_3=(\!a_1\x\!a_2)/|\!a_1\x\!a_2|$ is normal to $\t\Omega$.
The triple $\!a_i$ furnishes the covariant basis on $\t\Omega$.
The contravariant basis
$\!a^i$ is defined by the relations
$\!a^{\alpha}\cdot\!a_{\beta}=\delta^{\alpha}_{\beta}$ and $\!a^3=\!a_3$,
in which $\delta^{\alpha}_{\beta}$ is the Kronecker delta.
It is obvious that $\!a^{\alpha}$ are also tangent to the surface.
The metric tensor has the covariant components
$a_{\alpha\beta}=\!a_{\alpha}\cdot\!a_{\beta}$,  the determinant of
which is denoted by $a$. The contravariant components
are given by
$a^{\alpha\beta}=\!a^{\alpha}\cdot\!a^{\beta}$.
The curvature tensor
has covariant components
$b_{\alpha\beta}=\!a_3\cdot\partial_{\beta}\!a_{\alpha}$, whose
mixed components are $b^{\alpha}_{\beta}=a^{\alpha\gamma}b_{\gamma\beta}$.
The symmetric tensor $c_{\alpha\beta}=b^\gamma_\alpha b_{\gamma\beta}$ is called the third 
fundamental form of the surface.
The Christoffel symbols
are defined by
$\Gamma^{\gamma}_{\alpha\beta}
=\!a^{\gamma}\cdot\partial_{\beta}\!a_{\alpha}$,
which are symmetric with respect to the subscripts. The derivative of a scalar is a covariant vector.
The covariant derivative of a vector or tensor is a higher order tensor.
The formulas below will be used in the following.
\begin{equation}\label{covariant-derivative}
\begin{gathered}
u_{\alpha|\beta}=\partial_{\beta}u_{\alpha}-\Gamma^{\gamma}_{\alpha\beta}
u_{\gamma},\quad
(\partial_\alpha w)|_\beta=\partial_{\alpha\beta}w-\Gamma^{\lambda}_{\alpha\beta}\partial_{\lambda}w,\\
\sigma^{\alpha\beta}|_{\gamma}=\partial_{\gamma}\sigma^{\alpha\beta}
+\Gamma^{\alpha}_{\gamma\lambda}\sigma^{\lambda\beta}
+\Gamma^{\beta}_{\gamma\tau}\sigma^{\alpha\tau},\quad
b^{\gamma}_{\alpha|\beta}=\partial_{\beta}b^{\gamma}_{\alpha}
+\Gamma^{\gamma}_{\lambda\beta}b^{\lambda}_{\alpha}-
\Gamma^{\tau}_{\alpha\beta}b^{\gamma}_{\tau},\\
\rho_{\alpha\beta|\gamma}=\partial_{\gamma}\rho_{\alpha\beta}
-\Gamma^{\lambda}_{\gamma\alpha}\rho_{\lambda\beta}
-\Gamma^{\tau}_{\gamma\beta}\rho_{\alpha\tau},\\
\rho_{\alpha\beta|\gamma\delta}=\partial_\delta\rho_{\alpha\beta|\gamma}
-
\Gamma^\tau_{\alpha\delta}\rho_{\tau\beta|\gamma}
-
\Gamma^\tau_{\beta\delta}\rho_{\alpha\tau|\gamma}
-
\Gamma^\tau_{\gamma\delta}\rho_{\alpha\beta|\tau}.
\end{gathered}
\end{equation}
Product rules for differentiations, like
$(\sigma^{\alpha\lambda}u_{\lambda})|_{\beta}=
\sigma^{\alpha\lambda}|_{\beta}u_{\lambda}
+\sigma^{\alpha\lambda}u_{\lambda|\beta}$,
are valid. For more information see \cite{GZ}.

The mapping $\!phi$ is a one-to-one correspondence between $\Omega$ and $\t\Omega$. It maps a subdomain $\tau\subset\Omega$
to a subregion $\t\tau=\!phi(\tau)\subset\t\Omega$. 
A function $f$ defined on the shell middle surface will be identified with a function defined on $\Omega$ through the mapping $\!phi$ and denoted 
by the same notation. Thus $f(\!phi(x_\alpha))=f(x_\alpha)$. The integral over $\t\tau$ with respect to the surface area element 
is related to the double integral on $\tau$ by
\begin{equation*}
\int_{\t\tau}fd\t S=\int_\tau f\sqrt a dx_1dx_2.
\end{equation*}
We will ignore the area element $d\t S$ in the integral over the surface $\t\tau$, and 
simply write the left hand side integral as  $\displaystyle\int_{\t\tau}f$, and ignore the $dx_1dx_2$ in integral on the subdomain $\tau$, 
and write the right hand side integral as $\displaystyle\int_{\tau}f\sqrt a$.
The mapping $\!phi$ maps a curve $e\subset\overline\Omega$ to a curve $\t e=\!phi(e)$ contained in the closure of 
$\t\Omega$. Let $x_\alpha(s)$ be the arc length parameterization
of $e$, then $\!phi(x_\alpha(s))$ is a parameterization of $\t e$, but not in terms of the arc length of $\t e$. Let $\t s$ be the arc length 
parameter of $\t e$,  then the line integrals are related by
\begin{equation*}
\int_{\t e}fd\t s=\int_ef\sqrt{\sum_{\alpha,\beta=1, 2}a_{\alpha\beta}\frac{dx_\alpha}{ds}\frac{dx_\beta}{ds}}ds.
\end{equation*}
Similar to surface integrals, we will ignore the $d\t s$ in the left hand side line integral and the $ds$ in the right hand side line  integral.
For any line element $e$, area element $\tau$, and function $f$ that 
make the following integrals meaningful, we have 
\begin{equation*}
\int_{\t\tau}|f|\simeq\int_\tau|f|,\quad
\int_{\t e}|f|\simeq\int_e|f|.
\end{equation*}

\subsection{Koiter shell model}
The Koiter shell model \cite{Koiter} uses displacement of the shell mid-surface as the primary variable. A displacement $u_\alpha\!a^\alpha+w\!a^3$
deforms the surface $\t\Omega$ and changes the its curvature and metric tensors. The linearized change in curvature tensor is the bending
strain tensor, and linearized change of metric tensor is the membrane strain tensor. They, respectively, are
\begin{equation}\label{K-curvature}
\rho_{\alpha\beta}(\!u, w)=\partial^2_{\alpha\beta}w-\Gamma^{\gamma}_{\alpha\beta}
\partial_{\gamma}w+b^{\gamma}_{\alpha|\beta}u_{\gamma}+
b^{\gamma}_{\alpha}u_{\gamma|\beta}+b^{\gamma}_{\beta}u_{\gamma|\alpha}-c_{\alpha\beta}w,
\end{equation}
\begin{equation}\label{K-metric}
\gamma_{\alpha\beta}(\!u,w)=
\frac12(u_{\alpha|\beta}+u_{\beta|\alpha})
-b_{\alpha\beta}w.
\end{equation}

The loading forces on the shell body and upper and lower surfaces
enter the shell model as resultant loading forces per unit area  on the shell middle surface,
of which the tangential force density is
$p^{\alpha}\!a_{\alpha}$ and transverse force density $p^3\!a_3$.
Let the boundary $\partial\t\Omega$ be divided to $\partial^D\t\Omega\cup\partial^S\t\Omega\cup\partial^F\t\Omega$.
On $\partial^D\t\Omega$ the shell is clamped, on $\partial^S\t\Omega$ the shell is simply supported, and
on $\partial^F\t\Omega$ the shell free of displacement constraint and subject to force only.
(There are $16$ different ways to specify boundary conditions at any point on the shell boundary, of which we consider the three most 
frequently studied.)
The shell model is
defined in the Hilbert space
\begin{multline}\label{K-space}
H=\{(\!v, z)\in \!H^1\x H^2\ | v_\alpha \text{ and } z \text{ are }0\  \text{on}\ \partial^D\Omega\cup\partial^S\Omega, \\
\text{ and the normal derivative of  }z\text{ is }0\
\text{on}\ \partial^D\Omega\}.
\end{multline}
The model determines a unique $(\!u\e, w\e)\in H$
such that
\begin{multline}\label{K-model}
\frac13\int_{\t\Omega}
a^{\alpha\beta\lambda\gamma}\rho_{\lambda\gamma}(\!u\e, w\e)
\rho_{\alpha\beta}
(\!v, z)
+\eps^{-2}\int_{\t\Omega}
a^{\alpha\beta\lambda\gamma}\gamma_{\lambda\gamma}(\!u\e,w\e)
\gamma_{\alpha\beta}(\!v,z)\\
=
\int_{\t\Omega}
(p^{\alpha}v_{\alpha}+
p^3z)
+\int_{\partial^S\t\Omega}mD_{\!n}z
+\int_{\partial^F\t\Omega}\left(q^\alpha v_\alpha+q^3z+mD_{\!n}z\right)
\ \ \forall\
(\!v,z) \in H.
\end{multline}
Here, $q^i$ and $m$ are resultant loading functions on the shell boundary, which can be calculated from force resultants and
moment resultants on the shell edge \cite{Koiter}. The scalar $z$ can be viewed as
defined on $\t\Omega$. We let $\!n=n^\alpha\!a_\alpha$ be the unit outward normal to $\partial\t\Omega$ that is tangent to
$\t\Omega$. The derivative $D_{\!n}z=n^\alpha\partial_\alpha z$ is the directional derivative  in the direction of $\!n$ with respect to arc length.
The fourth order contravariant tensor
$a^{\alpha\beta\gamma\delta}$ is the elastic tensor of the shell,
defined by
\begin{equation*}
a^{\alpha\beta\gamma\delta}=\mu (a^{\alpha\gamma}a^{\beta\delta}+a^{\beta\gamma}a^{\alpha\delta})+
\frac{2\mu\lambda}{2\mu+\lambda}
a^{\alpha\beta}a^{\gamma\delta}.
\end{equation*}
Here, $\lambda$ and $\mu$ are the Lam\'e coefficients of the elastic material, which we assume to be constant.
The compliance tensor of the shell defines the inverse operator of the elastic tensor, given by
\begin{equation*}
a_{\alpha\beta\gamma\delta}=\frac{1}{2\mu}\left[
\frac12(a_{\alpha\delta}a_{\beta\gamma}+
a_{\beta\delta}a_{\alpha\gamma})-\frac{\lambda}{2\mu+3\lambda}a_{\alpha\beta}a_{\gamma\delta}
\right]
\end{equation*}
For symmetric tensors $\sigma^{\alpha\beta}$ and $\gamma_{\alpha\beta}$,
$\sigma^{\alpha\beta}=a^{\alpha\beta\gamma\delta}\gamma_{\gamma\delta}$ if and only if
$\gamma_{\alpha\beta}=a_{\alpha\beta\gamma\delta}\sigma^{\gamma\delta}$.
The elastic tensor is a continuous and positive definite operator in the sense that there is a positive constant $C$ depending
on the shell surface and shell material such that for any covariant tensors $\gamma_{\alpha\beta}$ and $\rho_{\alpha\beta}$,
\begin{equation}\label{elastic-tensor-equiv}
\begin{gathered}
a^{\alpha\beta\gamma\delta}\gamma_{\alpha\beta}\rho_{\gamma\delta}\le C\left(\sum_{\alpha, \beta=1,2}\gamma_{\alpha\beta}^2\right)^{1/2}
\left(\sum_{\alpha, \beta=1,2}\rho_{\alpha\beta}^2\right)^{1/2},\\
\sum_{\alpha, \beta=1,2}\gamma_{\alpha\beta}^2\le C a^{\alpha\beta\gamma\delta}\gamma_{\alpha\beta}\gamma_{\gamma\delta}.
\end{gathered}
\end{equation}
The compliance tensor has the similar property that for any contravariant tensors $\M^{\alpha\beta}$ and $\N^{\alpha\beta}$,
\begin{equation}\label{compliance-tensor-equiv}
\begin{gathered}
a_{\alpha\beta\gamma\delta}\M^{\alpha\beta}\N^{\gamma\delta}\le C\left(\sum_{\alpha, \beta=1,2}{\M^{\alpha\beta}}^2\right)^{1/2}
\left(\sum_{\alpha, \beta=1,2}{\N^{\alpha\beta}}^2\right)^{1/2},\\
\sum_{\alpha, \beta=1,2}{\M^{\alpha\beta}}^2\le C a_{\alpha\beta\gamma\delta}\M^{\alpha\beta}\M^{\gamma\delta}.
\end{gathered}
\end{equation}
The model  \eqref{K-model} has a unique solution in the space $H$ \cite{BCM}. When $\eps\to 0$, its solution
behaves in very different manners, depending on whether it is bending dominated, membrane dominated, or intermediate.
For bending dominated shell problems, when the resultant loading functions $p^\alpha$ and $p^3$ are independent of $\eps$,
the  model solution converges to a nonzero limit that solves the limiting bending model.
We show below that the scaled membrane stress
also converges to  a limit.

As did in \cite{ABrezzi2}, to make the analysis of the finite element method easier,
we split a small portion of the membrane part and add it to the bending part,
replace $\eps^{-2}-\frac13$  by $\eps^{-2}$,  introduce the scaled membrane stress
$\M^{\eps\alpha\beta}=
\eps^{-2}a^{\alpha\beta\lambda\gamma}\gamma_{\lambda\gamma}(\!u\e, w\e)$
as a new variable, and
write the model in a mixed form, in which the functions $(\!u\e, w\e)\in H$ and $\M^{\eps\alpha\beta}\in V$
(that is the space of tensors whose components are functions in
$L^2$) are determined as the solution of
\begin{equation}\label{K-P-model}
\begin{gathered}
\frac13\int_{\t\Omega}
\left[a^{\alpha\beta\lambda\gamma}\rho_{\lambda\gamma}(\!u\e, w\e)
\rho_{\alpha\beta}(\!v, z)
+
a^{\alpha\beta\lambda\gamma}\gamma_{\lambda\gamma}(\!u\e,w\e)
\gamma_{\alpha\beta}(\!v,z)\right]
+\int_{\t\Omega}
\M^{\eps\alpha\beta}
\gamma_{\alpha\beta}(\!v,z)\hfill
\\
\hfill
=\int_{\t\Omega}
(p^{\alpha}v_{\alpha}+
p^3z)
+\int_{\partial^S\t\Omega}mD_{\!n}z
+\int_{\partial^F\t\Omega}\left(q^\alpha v_\alpha+q^3z+mD_{\!n}z\right)
\ \ \forall\
(\!v,z) \in H,
\\
\int_{\t\Omega}\N^{\alpha\beta}\gamma_{\alpha\beta}(\!u\e, w\e)-
\eps^2\int_{\t\Omega}a_{\alpha\beta\lambda\gamma}\M^{\eps\alpha\beta}\N^{\lambda\gamma}=0\hfill \forall\ \N\in V.
\end{gathered}
\end{equation}
This mixed model  is the basis for the finite element method.
In the next subsection, we present a theory in an abstract format, which is applicable to the Koiter model in either the original form \eqref{K-model}
or the mixed form \eqref{K-P-model}. The results will also be used in analysis of the finite element model.

\subsection{Asymptotic estimates on the shell model}
\label{abstract}
Notations in this sub-section are independent of the rest of the paper.
The Koiter shell model \eqref{K-model} can be fitted in the abstract equation \eqref{prob1as} below.
Let $H$,
$U$, and $V$ be Hilbert spaces, $A$ and $B$ be
linear continuous operators from $H$ to
$U$ and $V$, respectively.
We assume
\begin{equation}\label{equiva1as}
\|Au\|_U+\|Bu\|_V\simeq\|u\|_H\ \ \forall \ u\in H.
\end{equation}
For any $\eps>0$ and $f\in H^*$, the dual space of $H$, there is a unique $u\e\in H$, such that
\begin{equation}\label{prob1as}
(Au\e,Av)_U+\eps^{-2}(Bu\e,Bv)_V=\langle f,v\rangle
\quad \forall\ v\in H.
\end{equation}
We let $\ker B\subset H$ be the kernel of the operator $B$, and
let $W\subset V$ be the range of $B$. We define a norm on $W$ by $\|w\|_{W}=\inf_{v\in H, Bv=w}\|v\|_U$ $\forall\ w\in W$, such that
$W$ is isomorphic to $H/\ker B$.
We let $\overline W$ be the closure of $W$ in $V$.
Thus $W$ is a dense subset of $\overline W$, and  $(\overline W)^*$ is dense  in $W^*$.
We need a weaker norm on $W$. For $w\in W$, we define $\|w\|_{\overline{\overline W}}=\|\pi_{\overline W}w\|_{W^*}$. Here $\pi_{\overline W}:\overline W\to(\overline W)^*$
is the inverse of Riesz representation. The relations among these norms is that for any $w\in W$,
$\|w\|_{\overline{\overline W}}\le \|\pi_{\overline W}w\|_{(\overline W)^*}=\|w\|_{\overline W}=\|w\|_V\le \|w\|_W$.
We let $\overline{\overline W}$ be the closure of $W$ in this new norm. This closure is isomorphic to $W^*$. We let $j_{[W^*\to \overline{\overline W}]}$
be the isomorphic mapping from $W^*$ to $\overline{\overline W}$.
We assume that $f|_{\ker B}\ne 0$, such that the limiting problem
\begin{equation}\label{limitas}
(Au^0, Av)_U=\langle f, v\rangle\ \ \forall\ v\in\ker B
\end{equation}
has a nonzero solution $u^0\in\ker B$.
\begin{thm}
For the solution of \eqref{prob1as},
we have the asymptotic behavior that $\lim_{\eps\to 0}\|u\e-u^0\|=0$. Furthermore, there is a unique $\M\in \overline{\overline W}$ such that
$\lim_{\eps\to 0}\|\eps^{-2}Bu\e-\M\|_{\overline{\overline W}}=0$.
\end{thm}
\begin{proof}
In view of equation \eqref{limitas}, we have $\langle f,v\rangle-(Au^0, Av)_U=0$ $\forall v\in\ker B$. Thus there is a unique $\zeta\in W^*$ such that
$\langle f,v\rangle-(Au^0, Av)_U=\langle\zeta, Bv\rangle$.
Subtracting $(Au^0, Av)$ from both sides of the equation \eqref{prob1as} and using the fact that $Bu^0=0$, we have
\begin{equation*}
\eps^2(A(u\e-u^0),Av)_U+(B(u\e-u^0),Bv)_{\overline W}=\eps^2\langle\zeta, Bv\rangle
\quad \forall\ v\in H.
\end{equation*}
This problem is in the form that was analyzed in  \cite{Caillerie} and \cite{CR1}. By
Theorem 2.1 of \cite{CR1},  we have
\begin{equation*}
\|A(u\e-u^0)\|_U+\eps^{-1}\|Bu\e\|_{V}
+\|\eps^{-2}\pi_{\overline W}Bu\e-\zeta\|_{W^*}
\simeq\|\zeta\|_{W^*+\eps(\overline W)^*},\\
\end{equation*}
The conclusion of the theorem them follows from the fact that
$\lim_{\eps\to 0}\|\zeta\|_{W^*+\eps(\overline W)^*}=0$
\cite{BL}, and
$\|\eps^{-2}\pi_{\overline W}Bu\e-\zeta\|_{W^*}
=\|\eps^{-2}Bu\e-\M\|_{\overline{\overline W}}$. Here $\M=j_{[W^*\to\overline{\overline W}]}\zeta$.
\end{proof}
In terms of the Koiter model \eqref{K-model}, the operator $B$ is the membrane strain operator,
and $\ker B$ is the space of isometric displacements of pure bending. The situation of $\ker B\ne 0$ is that the shell allows pure bendings,
and the condition $f|_{\ker B}\ne 0$
means that the load on the shell indeed activates pure bending.  The convergence described in this theorem
means when $\eps\to 0$, $\!u\e$ converges to a limit in $H^1$, $w\e$ converges to a limit in $H^2$ and the scaled
membrane stress $\M^{\eps\alpha\beta}$ converges to a limit in a space that generally can not be described  in the usual sense of space of functions
or distributions. This is a minimum information one needs to have in order to conceive  a possibility to make the term 
$\sum_{\tau\in\T_h}
\sum_{\alpha,\beta=1}^2h^4_{\tau}\|\M^{\eps\alpha\beta}\|^2_{2,\tau}$ in the error estimate \eqref{estimate1} 
small, uniformly with respect to $\eps$, by a limited number of triangles.

The mixed form of the Koiter shell model \eqref{K-P-model} can be fitted in the abstract problem
\eqref{isomorphism-abs} below. 
Let $H, V$ be Hilbert spaces. Let $a(\cdot,\cdot)$ and $c(\cdot,\cdot)$ be bilinear forms on $H$ and $V$, and $b(\cdot,\cdot)$
be a bilinear form on $H\x V$. We assume that there is  a constant $C$ such that
\begin{equation}\label{isomorphism-condition}
\begin{gathered}
|a(u, v)|\le C\|u\|_H\|v\|_H, \quad C^{-1}\|u\|_H^2\le a(u, u)\ \forall\ u, v\in H,\\
|c(p, q)|\le C\|p\|_V\|q\|_V, \quad C^{-1}\|p\|_V^2\le c(p, p)\ \forall\ p, q\in V,\\
|b(v, q)|\le C\|v\|_H\|q\|_V\ \forall\ v\in H, q\in V.
\end{gathered}
\end{equation}
For  $f\in H^*$ and $g\in V^*$,
we seek $u\in H$ and $p\in V$ such that
\begin{equation}\label{isomorphism-abs}
\begin{gathered}
a(u, v)+b(v, p)=\langle f, v\rangle\ \ \forall\ v\in H,\\
b(u, q)-\eps^2c(p, q)=\langle g, q\rangle\ \ \forall\ q\in V.
\end{gathered}
\end{equation}
This problem  has a unique solution in the space $H\x V$ \cite{ABrezzi2},
for which we need an accurate estimate. Although this problem has been extensively studied in the literature \cite{Brezzi-book},
we were not able to find what we exactly need. So we include the theorem below.

For $v\in H$, there is a $l(v)\in V^*$ such that $\langle l(v), q\rangle =b(v, q)$ $\forall\ q\in V$.
We let $B(v)=i_Vl(v)\in V$, with $i_V:V^*\to V$ being the Riesz representation operator. Let $W$ be the range of $B$.
We define a weaker (semi) norm on $V$ by
\begin{equation}\label{isomorphism-weak}
|q|_{\overline V}=\sup_{v\in H}\frac{b(v,q)}{\|v\|_H}\ \ \forall\ q\in V.
\end{equation}
If $W$ is dense in $V$,
this is a  weaker norm. Otherwise, it is a semi-norm. 
Whether $W$ is dense in $V$ or not, we have the following
equivalence result.
\begin{thm}\label{isomorphism-thm}
There exist constants $C_\alpha$ that only depend on the constant in \eqref{isomorphism-condition} such that 
\begin{multline}\label{isomorphism-equiv}
\|u\|_H+|p|_{\overline V}+\eps\|p\|_V\le C_1
\sup_{v\in H, q\in V}\frac{a(u, v)+b(v, p)-
b(u, q)+\eps^2c(p, q)}{\|v\|_H+|q|_{\overline V}+\eps\|q\|_V}\\
\le C_2(\|u\|_H+|p|_{\overline V}+\eps\|p\|_V)
\ \ \forall\ u\in H, \ p\in V.
\end{multline}
\end{thm}
\begin{proof}
The second inequality is obvious.
To prove the first inequality, we first assume that $W$ is dense in $V$. 
Then  $V^*$ is dense in $W^*$, and  we have $(\eps V^*\cap W^*)^*=\eps^{-1}V+W$ \cite{BL}.
We also see from the definition \eqref{isomorphism-weak} that $|q|_{\overline V}=\|\pi_Vq\|_{W^*}$. Here $\pi_V:V\to V^*$ is the inverse of Riesz representation.
For any $(u, p)\in H\x V$, we let $(f, g)\in H^*\x V^*$ be the corresponding right hand side functional in
the equation \eqref{isomorphism-abs}. We write $u=u_1+u_2$ and $p=p_1+p_2$,
with $(u_1, p_1)$ solving
\begin{equation*} 
\begin{gathered}
a(u_1, v)+b(v, p_1)=\langle f, v\rangle\ \ \forall\ v\in H,\\
b(u_1, q)-\eps^2c(p_1, q)=0\ \ \forall\ q\in V,
\end{gathered}
\end{equation*}
while $(u_2, p_2)$ solves
\begin{equation*} 
\begin{gathered}
a(u_2, v)+b(v, p_2)=0\ \ \forall\ v\in H,\\
b(u_2, q)-\eps^2c(p_2, q)=\langle g, q\rangle\ \ \forall\ q\in V.
\end{gathered}
\end{equation*}
It can be shown that  $\|u_1\|_H+|p_1|_{\overline V}+\eps\|p_1\|_V\le C\|f\|_{H^*}$ \cite{ABrezzi2}. This is to say
\begin{equation}\label{iso-est1}
\|u_1\|_H+|p_1|_{\overline V}+\eps\|p_1\|_V\le C
\sup_{v\in H}\frac{|a(u, v)+b(v, p)|}{\|v\|_{H}}
\end{equation}
From the first equation of the above second system, we see that $|p_2|_{\overline V}\le \|u_2\|_H$.
Let $u_0\in H$ be an element such that $\|u_0\|_H=\|Bu_0\|_W$. We write
\begin{equation*}
\langle g, q\rangle=(i_V g, q)=(Bu_0, q)+(i_V g_1, q)=b(u_0, q)+\langle g_1, q\rangle.
\end{equation*}
We thus have
\begin{equation*} 
\begin{gathered}
a(u_2-u_0, v)+b(v, p_2)=-a(u_0, v)\ \ \forall\ v\in U,\\
b(u_2-u_0, q)-\eps^2c(p_2, q)=\langle g_1, q\rangle\ \ \forall\ q\in V.
\end{gathered}
\end{equation*}
Taking $v=u_2-u_0$ and $q=p_2$ in this equation, and sum, we have
\begin{equation*}
\|u_2-u_0\|_H^2+\eps^2\|p_2\|^2_V=-a(u_0, u_2-u_0)-(\eps^{-1}i_V g_1, \eps p_2).
\end{equation*}
Using Cauchy--Schwarz inequality, we get
\begin{equation*}
\|u_2-u_0\|_H+\eps\|p_2\|_V\le C \|u_0\|_H+\eps^{-1}\|i_V g_1\|_V.
\end{equation*}
Thus
\begin{equation*}
\|u_2\|_H+|p_2|_{\overline V}+\eps\|p_2\|_V\le C(\|u_0\|_H+\eps^{-1}\|i_V g_1\|_V)\le C
(\|Bu_0\|_W+\eps^{-1}\|i_V g_1\|_V).
\end{equation*}
Since this is valid for any decomposition of $g$, we get
\begin{multline}\label{iso-est2}
\|u_2\|_H+|p_2|_{\overline V}+\eps\|p_2\|_V\le C\|i_Vg\|_{W+\eps^{-1}V}=\|i_V g\|_{(W^*\cap\eps V^*)^*}\\=
\sup_{q\in V}\frac{\langle i _Vg, \pi_Vq\rangle}{\|\pi_Vq\|_{W^*}+\eps\|\pi_Vq\|_{V^*}}=
\sup_{q\in V}\frac{\langle g, q\rangle}{|q|_{\overline V}+\eps\|q\|_{V}}.
\end{multline}
It follows from \eqref{iso-est1} and \eqref{iso-est2} that when $W$ is dense in $V$ for any $(u, p)\in H\x V$ we have
\begin{equation}\label{iso-est3}
\|u\|_H+|p|_{\overline V}+\eps\|p\|_V\le C
\sup_{(v, q)\in H\x V}\frac{a(u, v)+b(v, p)
-b(u, q)+\eps^2c(p, q)}
{\|v\|_H+|q|_{\overline V}+\eps\|q\|_{V}}.
\end{equation}

If $W$ is not dense in $V$, we let $\overline W$ be the closure of $W$ in $V$, and
decompose $V$ orthogonally as $V=\overline W\oplus\overline W^\perp$. Then any $q\in V$ can be written as
$q=q_W+q_\perp$ such that $q_W\in\overline W$ and $q_\perp\in\overline W^\perp$. We have $b(v, q)=b(v, q_W)$ and $|q|_{\overline V}=|q_W|_{\overline V}$
and $\|q\|^2_V=\|q_W\|^2_V+\|q_\perp\|^2_V$.
For any $u\in H$ and $p\in V$, it follows from \eqref{iso-est3} that
\begin{equation*}
\|u\|_H+|p_W|_{\overline V}+\eps\|p_W\|_V
\le C
\sup_{v\in H, q_W\in \overline W}\frac{a(u, v)+b(v, p_W)-
b(u, q_W)+\eps^2c(p_W, q_W)}{\|v\|_H+|q_W|_{\overline V}+\eps\|q_W\|_V}.
\end{equation*}
Thus
\begin{multline*}
\|u\|_H+|p|_{\overline V}+\eps\|p\|_V\le C
(\|u\|_H+|p_W|_{\overline V}+\eps\|p_W\|_V+\eps\|p_\perp\|_V)\\
\le C
\sup_{v\in H, q_W\in \overline W}\frac{a(u, v)+b(v, p_W)-
b(u, q_W)+\eps^2c(p_W, q_W)}{\|v\|_H+|q_W|_{\overline V}+\eps\|q_W\|_V}+C\eps\|p_\perp\|_V\\
\le C
\sup_{v\in H, q_W\in\overline W, q_\perp\in \overline W^\perp}\frac{a(u, v)+b(v, p_W)-
b(u, q_W)+\eps^2c(p_W, q_W)+\eps^2c(p_\perp, q_\perp)}{\|v\|_H+|q_W|_{\overline V}+\eps\|q_W\|_V+\eps\|q_\perp\|_V}\\
\le C\sup_{v\in H, q\in V}\frac{a(u, v)+b(v, p)-
b(u, q)+\eps^2c(p, q)}{\|v\|_H+|q|_{\overline V}+\eps\|q\|_V}.
\end{multline*}
\end{proof}

\section{The finite element model}
\label{principle}
As mentioned in the introduction, we assume that the coordinate domain $\Omega$ is a polygon.
On $\Omega$, we introduce a triangulation $\T_h$
that is shape regular but not necessarily quasi-uniform.
We use $\E^0_h$ to denote both
the union of interior edges and the set of all interior edges.
The set of edges on the boundary $\partial\Omega$
is denoted by $\E^{\partial}_h$ that is divided as $\E^D_h\cup\E^S_h\cup \E^F_h$, corresponding to clamped, simply supported, and free
portions of the shell boundary. We let
$\E_h=\E^0_h\cup\E^{\partial}_h$. For a $e\in\E^h$, we use $h_e$ to denote its length. We use $\tilde\E_h=\!phi(\E_h)$ to denote the curvilinear
edges on the shell mid-surface $\t\Omega$ in a similar way. 
On $\Omega_h$, for any piecewise vectors $u_\alpha$ and $v_\alpha$, scalars $w$ and $z$, symmetric tensors
$\M^{\alpha\beta}$ and $\N^{\alpha\beta}$,  we define the following bilinear and linear forms.
\begin{multline}\label{form_ub_a}
\ub a(\!u, w; \!v, z)
=
\frac13\left\{\int_{\t\Omega_h}\left[a^{\alpha\beta\lambda\gamma}
\rho_{\lambda\gamma}(\!u, w)\rho_{\alpha\beta}(\!v, z)
+a^{\alpha\beta\lambda\gamma}
\gamma_{\lambda\gamma}(\!u, w)\gamma_{\alpha\beta}(\!v, z)\right] \right.\hfill
\end{multline}
\begin{multline*}
-\int_{\t\E^0_h}\left[2a^{\tau\beta\sigma\lambda}\lbrac\rho_{\sigma\tau}(\!v, z)\rbrac b^{\alpha}_{\beta}
+a^{\alpha\lambda\gamma\delta}\lbrac\gamma_{\gamma\delta}(\!v, z)\rbrac\right]
\lbra u_{\alpha}\rbra_{n_{\lambda}}\\
-\int_{\t\E^0_h}\left[2a^{\tau\beta\sigma\lambda}\lbrac\rho_{\sigma\tau}(\!u, w)\rbrac b^{\alpha}_{\beta}
+a^{\alpha\lambda\gamma\delta}\lbrac\gamma_{\gamma\delta}(\!u, w)\rbrac
\right] 
\lbra v_{\alpha}\rbra_{n_{\lambda}}
\end{multline*}
\begin{multline*}
-\int_{\t\E^0_h}
a^{\alpha\beta\lambda\gamma}\lbrac\rho_{\lambda\gamma}(\!v, z)\rbrac
\lbra\partial_\alpha w\rbra_{n_{\beta}}
-\int_{\t\E^0_h}
a^{\alpha\beta\lambda\gamma}\lbrac\rho_{\lambda\gamma}(\!u, w)\rbrac
\lbra\partial_\alpha z\rbra_{n_{\beta}}\\
+
\int_{\t\E^0_h}
a^{\alpha\beta\lambda\gamma}\lbrac\rho_{\lambda\gamma|\beta}(\!v, z)
\rbrac\lbra w\rbra_{n_{\alpha}}
+
\int_{\t\E^0_h}
a^{\alpha\beta\lambda\gamma}\lbrac\rho_{\lambda\gamma|\beta}(\!u, w)
\rbrac\lbra z\rbra_{n_{\alpha}}
\end{multline*}
\begin{multline*}
-\int_{\t\E^{S\cup D}}\left[2a^{\gamma\beta\sigma\tau}\rho_{\sigma\tau}(\!u, w)b^{\alpha}_{\beta}
{n_{\gamma}}
+a^{\alpha\beta\gamma\delta}\gamma_{\gamma\delta}(\!u, w){n_{\beta}}
\right]v_{\alpha}
\\
-\int_{\t\E^{S\cup D}}\left[2a^{\gamma\beta\sigma\tau}\rho_{\sigma\tau}(\!v, z)b^{\alpha}_{\beta}
{n_{\gamma}}
+a^{\alpha\beta\gamma\delta}\gamma_{\gamma\delta}(\!v, z){n_{\beta}}
\right]u_{\alpha}
\end{multline*}

\begin{multline*}
+
\int_{\t\E^{S\cup D}}
[a^{\alpha\beta\lambda\gamma}\rho_{\lambda\gamma|\beta}(\!u, w)
{n_{\alpha}}+ D_{\!s}(a^{\alpha\beta\lambda\gamma}\rho_{\lambda\gamma}(\!u, w)
n_{\beta}s_\alpha)]z
\\
+
\int_{\t\E^{S\cup D}}
[a^{\alpha\beta\lambda\gamma}\rho_{\lambda\gamma|\beta}(\!v, z)
{n_{\alpha}}+ D_{\!s}(a^{\alpha\beta\lambda\gamma}\rho_{\lambda\gamma}(\!v, z)
n_{\beta}s_\alpha)]w
\\
\left.
-\int_{\t\E^D}
a^{\alpha\beta\lambda\gamma}\rho_{\lambda\gamma}(\!u, w)
n_{\beta}n_\alpha D_{\!n}z
-\int_{\t\E^D}
a^{\alpha\beta\lambda\gamma}\rho_{\lambda\gamma}(\!v, u)
n_{\beta}n_\alpha D_{\!n}w\right\}
\end{multline*}
Here $\!n=n^\alpha\!a_\alpha=n_\alpha\!a^\alpha$ is the unit outward normal to $\partial\t\Omega$ and tangent to $\t\Omega$,
and  $\!s=s^\alpha\!a_\alpha=s_\alpha\!a^\alpha$ is the unit counterclockwise tangent vector to $\partial\t\Omega$.
The function $z$ is treated as a function defined on $\t\Omega$ and $D_{\!n}z=n^\alpha\partial_\alpha z$ is the derivative
with respect to the arc length in the direction of $\!n$, and $D_{\!s}z=s^\alpha\partial_\alpha z$ is the derivative
in the direction of $\!s$. An edge $\t e\in\t\E^0_h$ is shared by elements $\t\tau_1$ and $\t\tau_2$. Piecewise
functions may have different values on the two elements, and thus discontinuous on $\t e$. The notation
$\lbrac\rho_{\sigma\tau}(\!v, z)\rbrac$ represents the average of values of  $\rho_{\sigma\tau}(\!v, z)$
from the sides of $\t\tau_1$ and $\t\tau_2$.
Let $\!n_\delta=n^\alpha_\delta\!a_\alpha=n_{\delta\alpha}\!a^\alpha$ be the unit outward normal to $\t e$ viewed as boundary of $\t\tau_\delta$.
We have $\!n_1+\!n_2=0$, $n^\alpha_1+n^\alpha_2=0$, and  $n_{1\alpha}+n_{2\alpha}=0$,
and $\lbra u_{\alpha}\rbra_{n_{\lambda}}=(u_{\alpha})|_{\t\tau_1} n_{1\lambda}+(u_{\alpha})|_{\t\tau_2} n_{2\lambda}$ is the jump of $u_\alpha$ over the edge
$\t e$ with respect to $\!n$, etc. We add  some additional penalty terms on the 
inter-element discontinuity and on the clamped and simply supported portions of the
boundary, and define a symmetric
bilinear form $a(\!u, w;\!v, z)$ by
\begin{multline}\label{form_a}
a(\!u, w;\!v, z)=\ub a(\!u, w;\!v, z)\\
+\C\sum_{e\in\E^0_h}
\left(h^{-1}_e\int_{e}\sum_{\alpha=1,2}\lbra u_{\alpha}\rbra \lbra v_{\alpha}\rbra+h^{-3}_e\int_{e}\lbra w\rbra \lbra z\rbra
+h^{-1}_e\int_{e}\sum_{\alpha=1,2}\lbra \partial_\alpha w\rbra \lbra \partial_\alpha z\rbra \right)
\\
+\C\sum_{e\in\E^{S}_h\cup\E^D_h}
\left(h^{-1}_e\int_{e}\sum_{\alpha=1,2}u_{\alpha}v_{\alpha}+h^{-3}_e\int_{e}wz\right)
+\C\sum_{e\in\E^D_h}
\left(h^{-1}_e\int_{e}D_{\!n}w D_{\!n}z\right).
\end{multline}
The jump $\lbra u_\alpha\rbra$ is the absolute value of the difference 
in the values of $u_\alpha$ from the two sides of $e$. We also define the bilinear forms 
\begin{multline}\label{form_b}
b(\M; \!v,z)=
\int_{\t\Omega_h}\M^{\alpha\beta}\gamma_{\alpha\beta}(\!v, z)
-\int_{\t\E^0_h}\lbrac\M^{\alpha\beta}\rbrac\lbra v_{\alpha}\rbra_{n_{\beta}}
 -\int_{\t\E^{S}_h\cup\t\E^D_h}
\M^{\alpha\beta}{n_{\beta}}v_{\alpha},
\hfill
\end{multline}
\begin{multline}\label{form_c}
c(\M, \N)=\int_{\t\Omega_h}a^{\alpha\beta\gamma\delta}\M_{\gamma\delta}\N_{\alpha\beta}.\hfill
\end{multline}
We define a linear form
\begin{multline}\label{form_f}
\langle \!f; \!v, z\rangle=\int_{\t\Omega_h}
(p^{\alpha}v_{\alpha}+
p^3z)
+\int_{\t\E^S_h}mD_{\!n}z
+\int_{\t\E^F_h}\left(q^\alpha v_\alpha+q^3z+mD_{\!n}z\right).\hfill
\end{multline}
All these forms are well defined for piecewise functions that could be independently defined on each element of the triangulation $\T_h$.

The finite element model is defined on a space of piecewise polynomials. We use continuous piecewise linear polynomials for the components
$\M^{\alpha\beta}$
of the scaled membrane stress, use discontinuous piecewise cubic  polynomials for the transverse deflection $w$, and use discontinuous
piecewise quadratic polynomials for components $u_\alpha$. The latter needs to be enriched  on elements that have one or two edges
on the free boundary $\E^F_h$. On an element $\tau$ (with edges $e_i$), we let $P^k(\tau)$ be the space
of polynomials of degree $k$.  If $\tau$ has one edge ($e_1$) on the free boundary,
we define two cubic polynomials $p^3_\alpha$ by
\begin{equation*}
p^3_1=\lambda_1p^2_1+1,\quad p^3_2=\lambda_1p^2_2+\lambda_2.
\end{equation*}
Here $\lambda_i$ are the barycentric coordinates with respect to the vertex opposite to the edge $e_i$, and $p^2_\alpha\in P^2(\tau)$ are
defined by
\begin{equation}\label{P3*}
\int_\tau(\lambda_1p^2_1+1)q\sqrt a=0\ \forall\ q\in P^2(\tau), \quad
\int_\tau(\lambda_1p^2_2+\lambda_2)q\sqrt a=0\ \forall\ q\in P^2(\tau).
\end{equation}
Note that functions in $\text{span}(p^3_\alpha)$ are orthogonal to $P^2(\tau)$ with respect to the inner product of 
$L^2(\tau)$ weighted by $\sqrt a$,
and they are linear on $e_1$, on which they can be determined by their zero and first moments weighted by $\sqrt a$.
We then define the $P^3_*(\tau)=P^2(\tau)\oplus\text{span}(p^3_\alpha)$.
If $\tau$ has two edges ($e_\alpha$) on the free boundary, we take the full $P^3(\tau)$.
A function in $P^3(\tau)$ is uniquely determined by its projection into  $P^2(\tau)$ with respect to $L^2(\tau)$ weighted by $\sqrt a$
and its zero and first moments on $e_1$ and $e_2$ weighted by $\sqrt a$. These will be explicitly given by the formulas \eqref{wI} -- \eqref{uI-edge2} below.
Let $P^u(\tau)= P^2(\tau)$, $P^3_*(\tau)$, or $P^3(\tau)$, depending on
whether  $\tau$ has no edge, one edge, or two edges on the free boundary $\E^F_h$.
The finite element space is defined by
\begin{equation}\label{FE-space}
\begin{gathered}
\H_h=\{(\!v, z); \text{ on each }\tau\in\T_h, v_\alpha\in P^u(\tau), z\in P^3(\tau)\},\\
\V_h=\{\N; \N^{\alpha\beta}\in H^1, \text{ on each }\tau\in\T_h,  \N^{\alpha\beta}\in P^1(\tau)\}.
\end{gathered}
\end{equation}
The finite element model seeks $(\!u, w)\in \H_h$ and $\M\in \V_h$ such that
\begin{equation}\label{K-fem}
\begin{gathered}
a(\!u, w; \!v, z)+b(\M; \!v, z)=\langle\!f;\!v, z\rangle\ \ \forall\ (\!v, z)\in \H_h, \\
b(\N; \!u, w)-\eps^2c(\M,\N)=0\ \ \forall\ \N\in \V_h.
\end{gathered}
\end{equation}
This equation is in the form of \eqref{isomorphism-abs}. We shall define the norms in $\H_h$ and $\V_h$ later.
We shall prove that the finite element model \eqref{K-fem} is well posed
if the penalty constant $\C$ in \eqref{form_a} is sufficiently large, by verifying the conditions \eqref{isomorphism-condition}.
This penalty constant $\C$ could be dependent on the shell geometry, the polynomial degrees, and the shape regularity $\K$ of
the triangulation $\T_h$. It is, otherwise, independent of the triangulation $\T_h$ and the shell thickness.

The solution of the Koiter model \eqref{K-P-model} satisfies the equation \eqref{K-fem} when the test functions $\!v$, $z$, $\N$
are arbitrary piecewise smooth functions, not necessarily polynomials. This says that the finite element model is consistent
with the shell model. The consistency is verified in the appendix.

Note that there is no boundary condition enforced on functions in the spaces $\H_h$ and $\V_h$.
The displacement boundary condition is enforced by boundary penalty in a consistent manner, which is Nitsche's method
\cite{A-DG}.
For tangential displacement on an element with one edge on $\E^F_h$, one may simply replace $P^3_*(\tau)$  by the richer $P^3(\tau)$.
This would slightly increase the complexity, but not affect the stability or accuracy of the finite element method.

\section{Korn's inequality on shells for piecewise functions }
\label{KornOnShell}
The wellposedness of the Koiter model \eqref{K-model}
is based on the following Korn type inequality on shells \cite{BCM}. There is a constant $C$ such that
\begin{multline}\label{Korn}
\sum_{\alpha=1,2}\|u_\alpha\|^2_{1,\Omega}+\|w\|^2_{2,\Omega}\le C
\left[\sum_{\alpha,\beta=1,2}\|\rho_{\alpha\beta}(\!u, w)\|^2_{0,\Omega}+\sum_{\alpha,\beta=1,2}\|\gamma_{\alpha\beta}(\!u, w)\|^2_{0,\Omega}+f^2(\!u, w)\right]
\\
\ \ \forall\
\!u\in \!H^1,\ w\in H^2.
\end{multline}
Here $f(\!u, w)$ is a semi-norm on $\!H^1\x H^2$ that satisfies the condition that
if $\!u, w$ defines a rigid body motion and $f(\!u, w)=0$ then $\!u=0$ and $w=0$.
(The displacement $u_\alpha\!a^\alpha+w\!a^3$ is a rigid body motion of the shell mid-surface
if and only if $\rho_{\alpha\beta}(\!u, w)=0$ and  $\gamma_{\alpha\beta}(\!u, w)=0$ \cite{BCM}.)
We need to generalize this inequality to piecewise functions on $\Omega_h$.
Let $H^1_h$ be the space of piecewise $H^1$ functions in which a function is independently defined on each element $\tau$, and
$u|_\tau\in H^1(\tau)$ for $\tau\in\T_h$. The space $H^2_h$ is
the space of piecewise $H^2$ functions defined similarly.
\begin{multline}\label{K-h-norm-u}
\|u\|_{H^1_h}^2:=\sum_{\tau\in\T_h}\|u\|^2_{1,\tau}
+\sum_{e\in \E^0_h}
h^{-1}_e\int_{e}\lbra u\rbra^2,\hfill
\end{multline}
\begin{multline}\label{K-h-norm-w}
\|w\|_{H^2_h}^2:=
\sum_{\tau\in\T_h}\|w\|^2_{2,\tau}
+\sum_{e\in \E^0_h}\left(\sum_{\alpha=1, 2}
h^{-1}_e\int_{e}\lbra \partial_\alpha w\rbra^2
+
h^{-3}_e\int_{e}\lbra w\rbra^2\right).\hfill
\end{multline}
For $u_\alpha\in H^1_h$ and $w\in H^2_h$, we define a norm
\begin{equation}\label{h-norm}
\|\!u, w\|_{\!H^1_h\x H^2_h}=\left(\sum_{\alpha=1,2}\|u_\alpha\|^2_{H^1_h}+\|w\|^2_{H^2_h}\right)^{1/2}.
\end{equation}
Let $f(\!u, w)$ be a semi-norm that is continuous with respect to this norm such that there is a $C$ only dependent of $\K$ of $\T_h$
and
\begin{equation}\label{f-continuous}
|f(\!u, w)|\le C\|\!u, w\|_{\!H^1_h\x H^2_h}\ \forall\ (\!u, w)\in \!H^1_h\x H^2_h.
\end{equation}
We also assume that $f$ satisfies the condition that
if $(\!u, w)\in \!H^1\x H^2$ defines a rigid body motion and $f(\!u, w)=0$ then $\!u=0$ and $w=0$.
We then define another norm on the space $\!H^1_h\x H^2_h$.
\begin{multline}\label{triple-norm}
\ll\!u, w\ll_h=
\left[
\sum_{\alpha,\beta=1,2}\|\rho_{\alpha\beta}(\!u, w)\|^2_{0,\Omega_h}+\sum_{\alpha,\beta=1,2}\|\gamma_{\alpha\beta}(\!u, w)\|^2_{0,\Omega_h}\right.
\\
\left.+\sum_{e\in \E^0_h}\left(
\sum_{\alpha=1,2}h^{-1}_e\int_{e}\lbra u_{\alpha}\rbra^2
+\sum_{\alpha=1, 2}
h^{-1}_e\int_{e}\lbra \partial_\alpha w\rbra^2
+
h^{-3}_e\int_{e}\lbra w\rbra^2\right)
+f^2(\!u, w)
\right]^{1/2}.
\end{multline}
We have the following generalization of the Korn's inequality to piecewise functions.
\begin{thm}\label{Korn-thm}
There exists a constant $C$ that could be dependent on the shell mid-surface,
and the shape regularity $\K$ of $\T_h$, but otherwise independent of the triangulation,
such that
\begin{equation}\label{Korn-inequality}
\|\!u, w\|_{\!H^1_h\x H^2_h}\le C\ll\!u, w\ll_h\ \forall\ w\in H^2_h,\ u_{\alpha}\in H^1_h.
\end{equation}
\end{thm}
In view of the definitions \eqref{K-curvature} and \eqref{K-metric}, this theorem implies that $\|\!u, w\|_{\!H^1_h\x H^2_h}\simeq \ll\!u, w\ll_h$.
To prove the inequality, we need a discrete Korn's inequality for piecewise functions in $H^1_h$, see (1.21) of \cite{Brenner-Korn}.
It says that
that there is a constant $C$ that might be dependent on the domain $\Omega$ and the shape regularity $\K$ of the triangulation $\T_h$, but
otherwise independent of $\T_h$ such that
\begin{equation}\label{Korn-Brenner}
\sum_{\alpha=1,2}\|u_\alpha\|^2_{H^1_h}\le C\left[\sum_{\alpha=1,2}\|u_\alpha\|^2_{0,\Omega_h}
+\sum_{\alpha, \beta=1,2}\|e_{\alpha\beta}(\!u)\|^2_{0,\Omega_h}+\sum_{e\in\E^0_h, \alpha=1,2}
h^{-1}_e\int_{e}\lbra u_{\alpha}\rbra^2\right].
\end{equation}
Here $e_{\alpha\beta}(\!u)=(\partial_\beta u_\alpha+\partial_\alpha u_\beta)/2$ is the symmetric part of the gradient of $\!u$.
We also need a trace theorem and a compact embedding result for functions in $H^1_h$.
\begin{lem}
Let $\tau$ be a triangle, and $e$ one of its edges. Then there is a
constant $C$ depending on the shape regularity of $\tau$ such that
\begin{equation}\label{trace}
\int_eu^2\le C\left[h_e^{-1}\int_{\tau}u^2+\sum_{\alpha=1,2}h_e\int_{\tau}|\partial_\alpha u|^2\right]
\ \ \forall\ u\in H^1(\tau).
\end{equation}
\end{lem}
This can be found in \cite{A-DG}.
For piecewise functions in $H^1_h$, we have a trace theorem that \cite{compact}
\begin{lem}\label{tracetheorem}
There
exists a constant $C$ depending on $\Omega$ and the shape regularity
$\K$ of $\T_h$, but otherwise independent of the triangulation such that
\begin{equation}\label{Omega-trace}
\|u\|_{L^2(\partial\Omega)}\le C \|u\|_{H^1_h}\ \ \forall\ u\in H^1_h.
\end{equation}
\end{lem}
The following compact embedding theorem can be derived from a result in \cite{DiPietro},
for which a direct proof can be found in \cite{compact}.
\begin{lem}\label{uniformcompactembedding}
Let $\T_{h_i}$ be a (infinite) class of shape regular but not necessarily quasi-uniform
triangulations of the polygonal domain $\Omega$, with a shape regularity constant $\K$.
For each $i$, let $H^1_{h_i}$ be the space of piecewise $H^1$ functions, subordinated to the
triangulation $T_{h_i}$,
equipped with the norm \eqref{K-h-norm-u}.
Let $\{u_i\}$ be a bounded sequence such that $u_i\in H^1_{h_i}$ for each $i$. I.e.,
there is a constant $C$, such that $\|u_i\|_{H^1_{h_i}}\le C$ for all $i$.
Then, the sequence $\{u_i\}$ has a convergent subsequence in $L^2$.
\end{lem}
\begin{proof}[Proof of Theorem~\ref{Korn-thm}]
It follows from the inequality \eqref{Korn-Brenner} and the definitions \eqref{K-curvature} and \eqref{K-metric}
of $\rho_{\alpha\beta}$ and $\gamma_{\alpha\beta}$
that there is a constant $C$ such that
\begin{equation}\label{Korn-thm-proof1}
\|\!u, w\|_{\!H^1_h\x H^2_h}^2\le C\left(\ll\!u, w\ll_h^2+\sum_{\alpha=1,2}\|u_\alpha\|^2_{0,\Omega_h}+\|w\|^2_{1,\Omega_h}\right)\  \forall\ w\in H^2_h\ u_{\alpha}\in H^1_h.
\end{equation}
On a fixed triangulation $\T_h$, it follows from the Rellich--Kondrachov
compact embedding theorem and Peetre's lemma (Theorem 2.1, page 18 in \cite{Raviart}) that
there is a constant $C_{\T_h}$ such that
\begin{equation*}
\|\!u, w\|_{\!H^1_h\x H^2_h}\le C_{\T_h}\ll\!u, w\ll_h\  \forall\ w\in H^2_h\ u_{\alpha}\in H^1_h.
\end{equation*}
We show that for a class  of shape regular triangulations, such $C_{\T_h}$ has a upper bound that only
depends on the shape regularity $\K$ of the whole class. Otherwise, there exists a sequence
of triangulations $\{\T_{h_n}\}$ and an associated sequence of functions $(\!u^n, w^n)$ in
$\!H^1_{h_n}\x H^2_{h_n}$
such that
\begin{equation*}
\|\!u^n, w^n\|_{\!H^1_{h_n}\x H^2_{h_n}}=1 \text{ and  }\ll\!u^n, w^n\ll_{h_n}\le 1/n.
\end{equation*}
It follows from Lemma~\ref{uniformcompactembedding}
that there is a subsequence,
still denoted by $(\!u^n, w^n)$ such that $\!u^n$ converges to $\!u^0\in\!L^2$,
the piecewise derivatives $\partial_\alpha w^n$  converges to  $w^0_\alpha \in L^2$,
and $w^n$ converges to $w^0\in L^2$. We show that this $w^0$ is actually in $H^1$ and we have that
$\partial_\alpha w^0=w^0_\alpha$.
To see this, we only need to show that the weak derivatives of $w^0$ are $(w^0_1, w^0_2)$.
For a compactly supported smooth function $\phi$, we have
\begin{equation*}
\int_{\Omega}w^0\partial_1\phi=\lim_{n\to\infty}\int_{\Omega}w^n\partial_1\phi.
\end{equation*}
For each $n$,
\begin{equation*}
\int_{\Omega}w^n\partial_1\phi=
-\int_{\Omega_{h_n}}\partial_1w^n\phi+\sum_{e\in\E^0_{h_n}}\int_e\lbra w^n\rbra_{\bar n_1}\phi.
\end{equation*}
Here, $\langle \bar n_1, \bar n_2\rangle$ is the unit normal to the edge $e$ and
$\lbra w^n\rbra_{\bar n_1}$ is the jump of $w^n$ over $e$ with respect to $\bar n_1$.
The second term in the right hand side can be bounded by
\begin{equation*}
C\left[\sum_{e\in\E^0_{h_n}}h^{-1}_e\int_e\lbra w^n\rbra^2\right]^{1/2}\left[|\phi|^2_{0,\Omega}+
\sum_{\tau\in\T_{h_n}}h^2_\tau|\phi|^2_{1,\tau}\right]^{1/2},
\end{equation*}
where $C$ depends on $\K$ only. Since $\ll\!u^n, w^n\ll_h\to 0$, this upper bound tends to zero as $n\to\infty$.
Thus we have
\begin{equation*}
\int_{\Omega}w^0\partial_1\phi=-\int_{\Omega}w^0_1\phi.
\end{equation*}
This shows that $w^0_1=\partial_1 w^0$ that is the weak derivative of $w^0$. Similarly, we have  $w^0_2=\partial_2 w^0$, and therefore $w^0\in H^1$.
We proved that
\begin{equation}\label{limit-u-w}
\lim_{n\to\infty}\|\!u^n-\!u^0\|_{L^2}=0 \text{ and }\lim_{n\to\infty}\left(\sum_{\tau\in\T_{h_n}}\|w^n-w^0\|^2_{1,\tau}\right)^{1/2}=0.
\end{equation}

Next we show  that $u^0_\alpha\!a^\alpha+w^0\!a^3$ is a rigid body motion of the shell mid-surface.
Let $RM=\{(\!u, w)\}$ be the space of components functions of rigid body motions, i.e.,
$u_{\alpha}\!a^\alpha+w\!a^3$ are rigid body motions. Let $q^i$ be a smooth surface force field that annihilates $RM$, i.e.,
\begin{equation*}
\int_{\t\Omega}(u_{\alpha}q^{\alpha}+wq^3)=0 \ \forall\ (\!u, w)\in RM.
\end{equation*}
Such a force field is an admissible loading on the shell with totally free boundary.
We let $m^{\alpha\beta}$ and $n^{\alpha\beta}$ be a stress resultant and a stress couple equilibrating $q^i$ such that
\begin{equation*}
\int_{\t\Omega}n^{\alpha\beta}\rho_{\alpha\beta}(\!v, z)+m^{\alpha\beta}\gamma_{\alpha\beta}(\!v, z)=
\int_{\t\Omega}(q^\alpha v_{\alpha}+q^3z)\ \ \forall\ (\!v, z)\in \!H^1\x H^2.
\end{equation*}
One can choose $m^{\alpha\beta}, n^{\alpha\beta}$ in the following manner.
We consider a shell with the mid surface $\t\Omega$, but of unit thickness $\eps=1$, loaded by
$q^i$, and free on its entire lateral boundary. The shell model has a unique solution $(\!u_1, w_1)$
in the quotient
space $(\!H^1\x H^2)/RM$. We then take
\begin{equation*}
n^{\alpha\beta}=\frac13a^{\alpha\beta\lambda\gamma}\rho_{\lambda\gamma}(\!u_1, w_1),\quad
m^{\alpha\beta}=a^{\alpha\beta\lambda\gamma}\gamma_{\lambda\gamma}(\!u_1, w_1).
\end{equation*}
Since this $(\!u_1, w_1)$ is the solution of a rather regular elliptic equation, under our assumption on the regularity of the shell model, 
we have $w_1\in H^4$ and $\!u_1\in \!H^3$.
Using the Green's theorem on surfaces, see the appendix, in view of the definitions \eqref{K-curvature} and \eqref{K-metric},
for an element $\tau\in\T_{h_h}$, we have
\begin{multline}\label{shell-equilibrium}
\int_{\t\tau}n^{\alpha\beta}\rho_{\alpha\beta}(\!v, z)+m^{\alpha\beta}\gamma_{\alpha\beta}(\!v, z)\\
=
\int_{\t\tau}\left(-m^{\alpha\beta}|_{\beta}-2n^{\lambda\gamma}|_{\gamma}b^{\alpha}_{\lambda}-
n^{\lambda\gamma}b^{\alpha}_{\lambda|\gamma}
\right)v_{\alpha}+
\int_{\t\tau}\left(n^{\alpha\beta}|_{\alpha\beta}-n^{\alpha\beta}c_{\alpha\beta}-m^{\alpha\beta}b_{\alpha\beta}
\right)z\\
+
\int_{\partial\t\tau}\left(2n^{\lambda\gamma}b^{\alpha}_{\lambda}n_{\gamma}+m^{\alpha\beta}n_{\beta}
\right)v_{\alpha}-
\int_{\partial\t\tau}n^{\alpha\beta}|_{\beta}n_{\alpha}z+
\int_{\partial\t\tau}n^{\alpha\beta}n_{\beta}\partial_{\alpha}z.
\end{multline}
This identity is also valid when $\tau$ is replaced by $\Omega$, and from that we get the following equilibrium equations and boundary conditions.
\begin{equation}\label{m-n-equilibrium}
\begin{gathered}
-m^{\alpha\beta}|_{\beta}-2n^{\lambda\gamma}|_{\gamma}b^{\alpha}_{\lambda}-
n^{\lambda\gamma}b^{\alpha}_{\lambda|\gamma}
=q^{\alpha},\quad
n^{\alpha\beta}|_{\alpha\beta}-n^{\alpha\beta}c_{\alpha\beta}-m^{\alpha\beta}b_{\alpha\beta}
=q^3\ \text{ in }\Omega,\\
2n^{\lambda\gamma}b^{\alpha}_{\lambda}n_{\gamma}+m^{\alpha\beta}n_{\beta}=0,\quad
-n^{\alpha\beta}|_{\beta}n_{\alpha}
-D_{\!s}(n^{\alpha\beta}n_{\beta}s_\beta)=0,\quad
D_{\!n}(n^{\alpha\beta}n_{\beta}n_\alpha)=0\ \text{ on }\partial\Omega.
\end{gathered}
\end{equation}
Since $(\!u^n, w^n)\to (\!u^0, w^0)$ in $L^2$, we have
\begin{equation*}
\int_{\t\Omega}(u^0_{\alpha}q^{\alpha}+w^0q^3)
=\lim_{n\to\infty}\int_{\t\Omega_{h_n}}(u^n_{\alpha}q^{\alpha}+w^nq^3).
\end{equation*}
For a given $n$, by using \eqref{shell-equilibrium} on each element of $\t\tau\subset\t\Omega_{h_n}$, and summing up, we get
\begin{multline*}
\int_{\t\Omega_{h_n}}(u^n_{\alpha}q^{\alpha}+w^nq^3)
\\
=\int_{\t\Omega_{h_n}}\left[u^n_{\alpha}
\left(-m^{\alpha\beta}|_{\beta}-2n^{\lambda\gamma}|_{\gamma}b^{\alpha}_{\lambda}-
n^{\lambda\gamma}b^{\alpha}_{\lambda|\gamma}\right)
+w^n
\left(n^{\alpha\beta}|_{\alpha\beta}-n^{\alpha\beta}c_{\alpha\beta}-m^{\alpha\beta}b_{\alpha\beta}\right)
\right]\hfill\\
=\sum_{\tau\in\T_{h_n}}
\int_{\t\tau}\left[n^{\alpha\beta}\rho_{\alpha\beta}(\!u^n, w^n)+m^{\alpha\beta}\gamma_{\alpha\beta}(\!u^n, w^n)\right] \hfill \\
\hfill -\sum_{\tau\in\T_{h_n}}\left[\int_{\partial\t\tau}
\left(2n^{\lambda\gamma}b^{\alpha}_{\lambda}n_{\gamma}+m^{\alpha\beta}n_{\beta}
\right)u^n_{\alpha}-
\int_{\partial\t\tau}n^{\alpha\beta}|_{\beta}n_{\alpha}w^n+
\int_{\partial\t\tau}n^{\alpha\beta}n_{\beta}\partial_{\alpha}w^n\right].
\end{multline*}
Using the boundary condition in \eqref{m-n-equilibrium}, we right the second line as
\begin{equation*}
-\sum_{e\in\E^0_{h_n}}\left[\int_{\t e}
\left(2n^{\lambda\beta}b^{\alpha}_{\lambda}+m^{\alpha\beta}
\right)\lbra u^n_{\alpha}\rbra_{n_\beta}-
\int_{\t e}n^{\alpha\beta}|_{\beta}\lbra w^n\rbra_{n_\alpha}+
\int_{\t e}n^{\alpha\beta}\lbra \partial_{\alpha}w^n\rbra_{n_\beta}\right].
\end{equation*}
We apply the trace estimate \eqref{trace} to each of the edges, and use Cauchy--Schwarz inequality,  to obtain
the following estimate.
\begin{multline*}
\left|\int_{\t\Omega_{h_n}}(u^n_{\alpha}q^{\alpha}+w^nq^3)\right|
\le C\\
\left[\sum_{\tau\in\T_{h_n}; \alpha, \beta=1,2}\left(| n^{\alpha\beta}|^2_{0,\tau}
+h^2_\tau|n^{\alpha\beta}|^2_{1,\tau}+h^4_\tau| n^{\alpha\beta}|^2_{2,\tau}
+| m^{\alpha\beta}|^2_{0,\tau}+h^2_\tau| m^{\alpha\beta}|^2_{1,\tau}\right)
\right]^{1/2}\ll \!u^n, w^n\ll_{h_n}.
\end{multline*}
Since $q^i$, $m^{\alpha\beta}$, and $n^{\alpha\beta}$ are independent of $n$,
and $\lim_{n\to\infty}\ll \!u^n, w^n\ll_{h_n}=0$, we have
$\int_{\t\Omega}(u^0_{\alpha}q^{\alpha}+w^0q^3)=0$. This is true for any smooth $q^i\in RM^\perp$. Therefore, $(\!u^0, w^0)\in RM$.

Finally, we show that $(\!u^0, w^0)=0$.
It follows from \eqref{Korn-thm-proof1}, \eqref{limit-u-w} and $\ll \!u^n, w^n\ll_{h_n}\to 0$ that
$\lim_{n\to\infty}\|\!u^n-\!u^0, w^n-w^0\|_{\!H^1_h\x H^2_h}=0$. Since $f$ is uniformly continuous
with respect to the norm $\|\cdot\|_{\!H^1_h\x H^2_h}$ and since $f(\!u^n, w^n)\to 0$ ($f$ is a part
in the triple norm), we see
$f(\!u^0, w^0)=0$. Thus $(\!u^0, w^0)=0$. Therefore, $\lim_{n\to\infty}\|\!u^n, w^n\|_{\!H^1_h\x H^2_h}=0$,
which is contradict to the assumption that $\|\!u^n, w^n\|_{\!H^1_h\x H^2_h}=1$.
\end{proof}


As an example, we take
\begin{equation*}
f(\!u, w)=
\left[\sum_{e\in\E^{S}_h\cup\E^D_h}
\left(\int_{e}\sum_{\alpha=1,2}u^2_{\alpha}+\int_{e}w^2\right)
+\sum_{e\in\E^D_h}\int_{e}(D_{\!n}w)^2\right]^{1/2}.
\end{equation*}
It follows from Lemma~\ref{tracetheorem} that there is a $C$ only dependent on $\K$ such that
the continuity condition \eqref{f-continuous} is satisfied by this $f$.
Under the assumption that  the measure of $\partial^D\Omega$ is positive,
it is verified in \cite{BCM} that if $(\!u, w)\in \!H^1\x H^2$ defines a rigid body motion and $f(\!u, w)=0$ then $\!u=0$ and $w=0$.
With this $f$ in the Korn's inequality \eqref{Korn-inequality}, 
we add boundary penalty term
\begin{equation*}
\sum_{e\in\E^{S}_h\cup\E^D_h}
\left(h^{-1}_e\int_{e}\sum_{\alpha=1,2}u^2_{\alpha}+h^{-3}_e\int_{e}w^2\right)
+\sum_{e\in\E^D_h}
h^{-1}_e\int_{e}(D_{\!n}w)^2
\end{equation*}
to the squares of both sides of \eqref{Korn-inequality}. We then have the equivalence that there is a constant $C$ that may be dependent
on the shape regularity $\K$ of the triangulation $\T_h$, but otherwise independent of the triangulation, such that
\begin{equation}\label{Hh-ah-equiv}
C^{-1}\|\!u, w\|_{a_h}\le \|\!u, w\|_{\H_h}\le C\|\!u, w\|_{a_h} \ \forall\ \!u\in \!H^1_h, w\in H^2_h.
\end{equation}
Here
\begin{multline}\label{Hh-norm}
\|\!u, w\|^2_{\H_h}:=\sum_{\tau\in\T_h}\left(\|u\|^2_{1,\tau}+\|w\|^2_{2,\tau}\right)\\
+\sum_{e\in \E^0_h}\left(
h^{-1}_e\int_{e}\sum_{\alpha=1,2}\lbra u_\alpha\rbra^2+\sum_{\alpha=1, 2}
h^{-1}_e\int_{e}\lbra \partial_\alpha w\rbra^2
+
h^{-3}_e\int_{e}\lbra w\rbra^2\right)\\
+
\sum_{e\in\E^{S}_h\cup\E^D_h}
\left(h^{-1}_e\int_{e}\sum_{\alpha=1,2}u^2_{\alpha}+h^{-3}_e\int_{e}w^2\right)
+\sum_{e\in\E^D_h}
h^{-1}_e\int_{e}(D_{\!n}w)^2,
\end{multline}
\begin{multline}\label{ah-norm}
\|\!u, w\|^2_{a_h}:=
\sum_{\tau\in\T_h}\sum_{\alpha,\beta=1,2}\left(\|\rho_{\alpha\beta}(\!u, w)\|^2_{0,\tau}+\|\gamma_{\alpha\beta}(\!u, w)\|^2_{0,\tau}\right)\\
+\sum_{e\in \E^0_h}\left(
h^{-1}_e\int_{e}\sum_{\alpha=1,2}\lbra u_\alpha\rbra^2+\sum_{\alpha=1, 2}
h^{-1}_e\int_{e}\lbra \partial_\alpha w\rbra^2
+
h^{-3}_e\int_{e}\lbra w\rbra^2\right)\\
+
\sum_{e\in\E^{S}_h\cup\E^D_h}
\left(h^{-1}_e\int_{e}\sum_{\alpha=1,2}u^2_{\alpha}+h^{-3}_e\int_{e}w^2\right)
+\sum_{e\in\E^D_h}
h^{-1}_e\int_{e}(D_{\!n}w)^2.
\end{multline}

\section{Error analysis of  the finite element method}
\label{ErrorAnalysis}
The finite element model defined by \eqref{form_a} to \eqref{K-fem} fits in the form of the mixed equation
\eqref{isomorphism-abs}. We verify the conditions \eqref{isomorphism-condition} for the bilinear forms
defined by \eqref{form_a}, \eqref{form_b}, and \eqref{form_c}, with the space defined by \eqref{FE-space},
in which the $\H_h$ norm is defined by \eqref{Hh-norm}. We define the $\V_h$ norm by
\begin{equation}\label{Vh-norm}
\|\N\|_{\V_h}:=\left(\sum_{\alpha,\beta=1,2}\|\N^{\alpha\beta}\|^2_{0,\Omega} \right)^{1/2}.
\end{equation}
We show that there is a constant $C$ that depends on the shell geometry and shape regularity $\K$ of the triangulation $\T_h$, but otherwise, independent of the
triangulation such that
\begin{align} 
|a(\!u, w; \!v, z)| &\le C\|\!u, w\|_{\H_h}\|\!v, z\|_{\H_h}&\  &\forall\ (\!u, w), (\!v, z)\in \H_h,   \label{a-condition1}      \\
\|\!v, z\|_{\H_h}^2& \le Ca(\!v, z; \!v, z)&\  &\forall\  (\!v, z)\in \H_h,\label{a-condition2} \\
|b(\N; \!v, z)| &\le C\|\!v, z\|_{\H_h}\|\N\|_{\V_h}&\  &\forall\ (\!v, z)\in \H_h, \N\in \V_h,\label{b-condition}\\
|c(\M, \N)|&\le C\|\M\|_{V_h}\|\N\|_{\V_h}&\  &\forall\ \M, \N\in \V_h,\label{c-condition1}\\
\|\N\|_{\V_h}^2&\le Cc(\N, \N)&\  &\forall\ \N\in \V_h. \label{c-condition2}
\end{align} 
We start with \eqref{a-condition1}.
From the definition \eqref{form_ub_a} and \eqref{form_a} of the bilinear form $a$, using the property of the elastic tensor \eqref{elastic-tensor-equiv},
we see the first line
in \eqref{form_ub_a} is bounded as
\begin{multline*}
\left|\int_{\t\Omega_h}\left[a^{\alpha\beta\lambda\gamma}
\rho_{\lambda\gamma}(\!u, w)\rho_{\alpha\beta}(\!v, z)
+a^{\alpha\beta\lambda\gamma}
\gamma_{\lambda\gamma}(\!u, w)\gamma_{\alpha\beta}(\!v, z)\right]\right|
\\
\le C\|\!u, w\|_{a_h}\|\!v, z\|_{a_h}\le C \|\!u, w\|_{\H_h}\|\!v, z\|_{\H_h}.
\end{multline*}
We then estimate the fifth line in \eqref{form_ub_a}.
Let $e\in\E^0_h$ be one of the interior edges shared by elements $\tau_1$ and $\tau_2$. Using the elastic tensor property \eqref{elastic-tensor-equiv},
the H\"older inequality,
the trace inequality \eqref{trace},  the formula \eqref{K-curvature} and \eqref{covariant-derivative} for $\rho_{\lambda\gamma|\beta}$, and
inverse inequality for finite element functions, we have
\begin{multline*}
\left|\int_{\t e}
a^{\alpha\beta\lambda\gamma}\lbrac\rho_{\lambda\gamma|\beta}(\!v, z)\rbrac\lbra w\rbra_{n_{\alpha}}\right|\le C
\left[\sum_{\lambda, \gamma,\beta=1,2}h^3_e\int_e\lbrac\rho_{\lambda\gamma|\beta}(\!v, z)\rbrac^2
\right]^{1/2}
\left[h^{-3}_e\int_e\lbra w\rbra^2
\right]^{1/2}\\
\le C
\left[\sum_{\lambda, \gamma,\beta, \delta, \alpha=1,2}\left(h^2_{\tau_\delta}\int_{\tau_\delta}|\rho_{\lambda\gamma|\beta}(\!v, z)|^2+
h^4_{\tau_\delta}\int_{\tau_\delta}|\partial_\alpha\rho_{\lambda\gamma|\beta}(\!v, z)|^2\right)
\right]^{1/2}
\left[h^{-3}_e\int_e\lbra w\rbra^2\right]^{1/2}
\\
\le C
\left[\sum_{\delta=1,2}\left(h^2_{\tau_\delta}
(\|z\|^2_{3,\tau_\delta}+\|\!v\|^2_{2,\tau_\delta})+
h^4_{\tau_\delta}
(\|z\|^2_{4,\tau_\delta}+\|\!v\|^2_{3,\tau_\delta})
\right)
\right]^{1/2}
\left[h^{-3}_e\int_e\lbra w\rbra^2\right]^{1/2}\\
\le C
\left[\sum_{\delta=1,2}\left(
\|z\|^2_{2,\tau_\delta}+\|\!v\|^2_{1,\tau_\delta}
\right)
\right]^{1/2}
\left[h^{-3}_e\int_e\lbra w\rbra^2\right]^{1/2}.
\end{multline*}
From this, we get the estimate on the fifth line in \eqref{form_ub_a} that
\begin{multline}\label{a-5-est}
\left|\int_{\t\E^0_h}
a^{\alpha\beta\lambda\gamma}\lbrac\rho_{\lambda\gamma|\beta}(\!v, z)
\rbrac\lbra w\rbra_{n_{\alpha}}
+
\int_{\t\E^0_h}
a^{\alpha\beta\lambda\gamma}\lbrac\rho_{\lambda\gamma|\beta}(\!u, w)
\rbrac\lbra z\rbra_{n_{\alpha}}\right|\\
\le C
\left[\sum_{\tau\in\T_h}(\|z\|^2_{2, \tau}+\|\!v\|^2_{1,\tau})
\right]^{1/2}
\left[\sum_{e\in\E^0_h}h^{-3}_e\int_e\lbra w\rbra^2\right]^{1/2}\\
+C\left[\sum_{\tau\in\T_h}(\|w\|^2_{2, \tau}+\|\!u\|^2_{1,\tau})
\right]^{1/2}
\left[\sum_{e\in\E^0_h}h^{-3}_e\int_e\lbra z\rbra^2\right]^{1/2}.
\end{multline}
The other lines in \eqref{form_ub_a} and the penalty terms in \eqref{form_a} can be bounded in a very same manner.
All these together with applications of Cauchy--Schwarz inequality proves \eqref{a-condition1}.

Next, we consider \eqref{a-condition2}.
Let $B(\!u, w;\!v, z)$ be a bilinear form
defined by the sum of all the lines but the first one in the definition \eqref{form_ub_a} of $\ub a(\!u, w;\!v, z)$.
In view of the equivalence \eqref{Hh-ah-equiv}, there are constants $C_1$ and $C_3$ that depend on the shell mid-surface
and $\K$, and $C_2$ that depend on the penalty constant $\C$ in \eqref{form_a} such that
\begin{multline*}
a(\!v, z;\!v, z)\ge
C_1\|\!v, z\|^2_{\H_h}+C_2
\left[\sum_{e\in \E^0_h}\left(
\sum_{\alpha=1,2}h^{-1}_e\int_{e}\lbra v_\alpha\rbra^2+\sum_{\alpha=1, 2}
h^{-1}_e\int_{e}\lbra \partial_\alpha z\rbra^2
+
h^{-3}_e\int_{e}\lbra z\rbra^2\right)\right.\\
\left.
+
\sum_{e\in\E^{S}_h\cup\E^D_h}
\left(\sum_{\alpha=1,2}h^{-1}_e\int_{e}v^2_{\alpha}+h^{-3}_e\int_{e}z^2\right)
+\sum_{e\in\E^D_h}
h^{-1}_e\int_{e}(D_{\!n}z)^2\right]
-
C_3|B(\!v, z; \!v, z)|
\end{multline*}
Using the same argument as in \eqref{a-5-est}, we have an upper bound that
\begin{multline}
|B(\!v, z; \!v, z)|\le C
\|\!v, z\|_{\H_h}
\left[\sum_{e\in \E^0_h}\left(\sum_{\alpha=1,2}
h^{-1}_e\int_{e}\lbra v_\alpha\rbra^2+\sum_{\alpha=1, 2}
h^{-1}_e\int_{e}\lbra \partial_\alpha z\rbra^2
+
h^{-3}_e\int_{e}\lbra z\rbra^2\right)\right.\\
\left.
+
\sum_{e\in\E^{S}_h\cup\E^D_h}
\left(\sum_{\alpha=1,2}h^{-1}_e\int_{e}v^2_{\alpha}+h^{-3}_e\int_{e}z^2\right)
+\sum_{e\in\E^D_h}
h^{-1}_e\int_{e}(D_{\!n}z)^2
\right]^{1/2}.
\end{multline}
It follows from Cauchy--Schwarz inequality that when the penalty constant $\C$ in \eqref{form_a} is sufficiently big (which makes
$C_2$ sufficiently big)
there is a $C$ such that \eqref{a-condition2} holds. 
The continuity \eqref{b-condition} is proved similarly as \eqref{a-condition1}, but simpler.
The conditions \eqref{c-condition1} and \eqref{c-condition2} are trivial consequences of \eqref{compliance-tensor-equiv}.

Thus the finite element model \eqref{K-fem} has a unique solution
in the finite element space \eqref{FE-space}.
Corresponding to the weak norm \eqref{isomorphism-weak}, we define a weaker (semi) $\overline\V_h$ norm for finite element function in $\V_h$
\begin{equation}\label{isomorphism-weak-N}
|\N|_{\overline\V_h}:=\sup_{(\!v, z)\in \H_h}\frac{b(\N; \!v, z)}{\|\!v, z\|_{\H_h}}\ \ \forall\ \N\in \V_h.
\end{equation}
We are now in a situation for which Theorem~\ref{isomorphism-thm} is applicable. From that theorem,
we have the inequality that there exists a $C$ that could be dependent on the shell mid-surface and the shape regularity $\K$
of the triangulation $\T_h$, but otherwise independent
of the finite element mesh and the shell thickness $\eps$ such that
\begin{multline*}
\|(\!u, w)\|_{\H_h}+|\M|_{\overline\V_h}+\eps\|\M\|_{\V_h}\\
\le C
\sup_{(\!v, z)\in\H_h, \N\in\V_h}
\frac{a(\!u, w;\!v, z)+b(\M; \!v, z)-b(\N; \!u, w)+\eps^2
c(\M, \N)} {\|(\!v, z)\|_{\H_h}+|\N|_{\overline\V_h}+\eps\|\N\|_{\V_h}}\\
 \ \forall\
(\!u, w)\in\H_h, \M\in\V_h.
\end{multline*}
Let $\!u\e, w\e,\M\e$ be the solution to the Koiter model \eqref{K-P-model}, let
$\!u^h, w^h, \M^h$ be the finite element solution to the finite element model \eqref{K-fem},
and let $\!u^I, w^I, \M^I$ be an interpolation to the Koiter model solution from the finite element space.
Since the finite element method \eqref{K-fem} and the Koiter model \eqref{K-P-model}
are consistent, we have
\begin{multline}\label{error-fraction}
\|(\!u^h-\!u^I, w^h-w^I)\|_{\H_h}+|\M^h-\M^I|_{\overline\V_h}+\eps\|\M^h-\M^I\|_{\V_h}\\
\le
\sup_{(\!v, z)\in\H_h, \N\in\V_h}
\frac{\left[\begin{gathered}
a(\!u^h-\!u^I, w^h-w^I;\!v, z)+b(\M^h-\M^I; \!v, z)\\
-b(\N; \!u^h-\!u^I, w^h-w^I)+\eps^2
c(\M^h-\M^I, \N)\end{gathered}
\right]} {\|(\!v, z)\|_{\H_h}+|\N|_{\overline\V_h}+\eps\|\N\|_{\V_h}}\\
=
\sup_{(\!v, z)\in\H_h, \N\in\V_h}
\frac{\left[\begin{gathered}
a(\!u\e-\!u^I, w\e-w^I;\!v, z)+b( \M\e-\M^I; \!v, z)\\
-b(\N; \!u\e-\!u^I, w\e-w^I)+\eps^2
c(\M\e-\M^I, \N)\end{gathered}
\right]} {\|(\!v, z)\|_{\H_h}+|\N|_{\overline\V_h}+\eps\|\N\|_{\V_h}}.
\end{multline}
We estimate the four terms in the numerator of above last line one by one.
\begin{lem}\label{a-error-lem}
There is a constant $C$ independent of $\T_h$ such that
\begin{multline}\label{a-error}
\left|a(\!u\e-\!u^I, w\e-w^I;\!v, z)\right|
\\
\le C\|(\!v, z)\|_{\H_h}\left[
\sum_{\tau\in\T_h}\left(
\sum_{k=0}^4h^{2k-4}_\tau|w\e-w^I|^2_{k, \tau}+\sum_{k=0}^3\sum_{\alpha=1,2}h^{2k-2}_\tau|u\e_{\alpha}-u^I_{\alpha}|^2_{k, \tau}\right)\right]^{1/2}
\ \forall\ (\!v, z)\in \H_h.
\end{multline}
\end{lem}
\begin{proof}
In view of the formulas \eqref{form_a} and \eqref{form_ub_a}, there are totally $26$ terms in the expression of $a(\!u\e-\!u^I, w\e-w^I;\!v, z)$.
The following inequalities provide bounds to various terms.
Using the Cauchy--Schwarz inequality, the equivalence \eqref{Hh-ah-equiv}, and the definition \eqref{ah-norm}, we get
\begin{multline}\label{term-1-2}
\left|\int_{\t\Omega_h}\left[a^{\alpha\beta\lambda\gamma}
\rho_{\lambda\gamma}(\!u\e-\!u^I, w\e-w^I)\rho_{\alpha\beta}(\!v, z)
+
a^{\alpha\beta\lambda\gamma}
\gamma_{\lambda\gamma}(\!u\e-\!u^I, w\e-w^I)\gamma_{\alpha\beta}(\!v, z)\right]\right|\\
\le C \left[\sum_{\tau\in\T_h}\left(\|w\e-w^I\|^2_{2,\tau}+\|\!u\e-\!u^I\|^2_{1,\tau}\right)\right]^{1/2}\|\!v, z\|_{\H_h}.
\end{multline}
Using the trace theorem \eqref{trace} and inverse inequality for finite element functions, we get
\begin{multline}\label{term-3}
\left|\int_{\t\E^0_h}\left[2a^{\tau\beta\sigma\lambda}\lbrac\rho_{\sigma\tau}(\!v, z)\rbrac b^{\alpha}_{\beta}
+a^{\alpha\lambda\gamma\delta}\lbrac\gamma_{\gamma\delta}(\!v, z)\rbrac\right]
\lbra u\e_{\alpha}-u^I_\alpha\rbra_{n_{\lambda}}\right|\\
\le C
\sum_{e\in\E^0_h}
\left[h_e\int_e  \lbrac\rho_{\alpha\beta}(\!v, z)\rbrac^2+\lbrac\gamma_{\alpha\beta}(\!v, z)\rbrac^2
 \right]^{1/2}
\left[
h^{-1}_e\int_e\lbra \!u\e-\!u^I\rbra^2
\right]^{1/2}\\
\le C
\left[\sum_{\tau\in\T_h}\left[\|z\|^2_{2,\tau}+\|\!v\|^2_{1,\tau}+h^2_\tau(|z|^2_{3,\tau}+|\!v|^2_{2,\tau})\right)
 \right]^{1/2}\\
\left[\sum_{\tau\in\T_h}
\left(h^{-2}_\tau|\!u\e-\!u^I|^2_{0,\tau}+|\!u\e-\!u^I|^2_{1,\tau}\right)
\right]^{1/2}\\
\le C
\left[\sum_{\tau\in\T_h}
\left(h^{-2}_\tau|\!u\e-\!u^I|^2_{0,\tau}+|\!u\e-\!u^I|^2_{1,\tau}\right)
\right]^{1/2}\|\!v, z\|_{\H_h}.
\end{multline}
\begin{multline}\label{term-4}
\left|\int_{\t\E^0_h}\left[2a^{\tau\beta\sigma\lambda}\lbrac\rho_{\sigma\tau}(\!u\e-\!u^I, w\e-w^I)\rbrac
b^{\alpha}_{\beta}
+a^{\alpha\lambda\gamma\delta}\lbrac\gamma_{\gamma\delta}(\!u\e-\!u^I, w\e-w^I)\rbrac
\right]
\lbra v_{\alpha}\rbra_{n_{\lambda}}\right|\\
\le C
\sum_{e\in\E^0_h}
\left[h_e\int_e  \lbrac\rho_{\alpha\beta}(\!u\e-\!u^I, w\e-w^I)\rbrac^2+
\lbrac\gamma_{\alpha\beta}(\!u\e-\!u^I, w\e-w^I)\rbrac^2
 \right]^{1/2}
\left[
h^{-1}_e\int_e\lbra v\rbra^2
\right]^{1/2}\\
\le C
\left[\sum_{\tau\in\T_h}\left[\|w\e-w^I\|^2_{2,\tau}+\|\!u\e-\!u^I\|^2_{1,\tau}+h^2_\tau\left(|w\e-w^I|^2_{3,\tau}
+|\!u\e-\!u^I|^2_{2,\tau}\right)\right]
 \right]^{1/2}\\
\left[\sum_{e\in\E^0_h}h^{-1}_e\int_e\lbra v\rbra^2 \right]^{1/2}.
\end{multline}
Similar to \eqref{term-3}, we have
\begin{multline}\label{term-5}
\left|\int_{\t\E^0_h}
a^{\alpha\beta\lambda\gamma}\lbrac\rho_{\lambda\gamma}(\!v, z)\rbrac
\lbra\partial_\alpha w\e-\partial_\alpha w^I\rbra_{n_{\beta}}\right|\\
\le C
\left[\sum_{\tau\in\T_h}
\left(h^{-2}_\tau|w\e-w^I|^2_{1,\tau}+|w\e-w^I|^2_{2,\tau}\right)
\right]^{1/2}\|\!v, z\|_{\H_h}.
\end{multline}
Similar to \eqref{term-4}, we have
\begin{multline}
\left|\int_{\t\E^0_h}
a^{\alpha\beta\lambda\gamma}\lbrac\rho_{\lambda\gamma}(\!u\e-\!u^I, w\e-w^I)\rbrac
\lbra\partial_\alpha z\rbra_{n_{\beta}}\right|
\le C
\left[\sum_{e\in\E^0_h}h^{-1}_e\int_e\lbra \partial_\alpha z\rbra^2 \right]^{1/2}\\
\left[\sum_{\tau\in\T_h}\left[\|w\e-w^I\|^2_{2,\tau}+\|\!u\e-\!u^I\|^2_{1,\tau}+h^2_\tau\left(|w\e-w^I|^2_{3,\tau}
+|\!u\e-\!u^I|^2_{2,\tau}\right)\right]
 \right]^{1/2}.
\end{multline}
We also have
\begin{multline}
\left|\int_{\t\E^0_h}
a^{\alpha\beta\lambda\gamma}\lbrac\rho_{\lambda\gamma|\beta}(\!v, z)
\rbrac\lbra w\e-w^I\rbra_{n_{\alpha}}\right|\\
\le C
\left[\sum_{\tau\in\T_h}
\left(h^{-4}_\tau|w\e-w^I|^2_{0,\tau}+h^{-2}_\tau|w\e-w^I|^2_{1,\tau}\right)
\right]^{1/2}\|\!v, z\|_{\H_h}.
\end{multline}
\begin{multline}
\int_{\t\E^0_h}
a^{\alpha\beta\lambda\gamma}\lbrac\rho_{\lambda\gamma|\beta}(\!u\e-\!u^I, w\e-w^I)
\rbrac\lbra z\rbra_{n_{\alpha}}\\
\le C
\left[
\sum_{\tau\in\T_h}\left(
\sum_{k=0}^4h^{2k-4}_\tau|w\e-w^I|^2_{k, \tau}+\sum_{k=0}^3\sum_{\alpha=1,2}h^{2k-2}_\tau|u\e_{\alpha}-u^I_{\alpha}|^2_{k, \tau}\right)\right]^{1/2}
\left[\sum_{e\in\E^0_h}h^{-3}_e\int_e\lbra z\rbra^2 \right]^{1/2}.
\end{multline}
The  boundary terms  in \eqref{form_ub_a}
are estimated using the same tools.
We then treat the interior penalty terms in \eqref{form_a}. Let $e\in\E^0_h$ be shared by $\tau_1$ and $\tau_2$. We use the H\"older inequality to get
\begin{multline*}
\sum_{\alpha=1,2}h^{-1}_e\int_{e}\lbra u\e_{\alpha}-u^I_\alpha\rbra \lbra v_{\alpha}\rbra
\le
\left[  h^{-1}_e\int_{ e}\lbra \!u\e-\!u^I\rbra^2
\right]^{1/2}
\left[  h^{-1}_e\int_{e}\lbra \!v\rbra^2
\right]^{1/2}\\
\le C
\left[\int_{\tau_1\cup\tau_2}\left(h^{-2}_\tau|\!u\e-\!u^I|^2_{0,\tau_1\cup\tau_2}+|\!u\e-\!u^I|^2_{1,\tau_1\cup\tau_2}\right)
\right]^{1/2}
\left[  h^{-1}_e\int_{e}\lbra \!v\rbra^2
\right]^{1/2}.
\end{multline*}
Similar estimates can be established for 
\begin{equation*}
h^{-3}_e\int_{e}\lbra w\e-w^I\rbra \lbra z\rbra,\quad
\sum_{\alpha=1,2}h^{-1}_e\int_{\e}\lbra \partial_\alpha w\e-\partial_\alpha w^I\rbra \lbra \partial_\alpha z\rbra.
\end{equation*}
From these, we obtain
\begin{multline}
\left|\sum_{e\in\E^0_h}
\left(\sum_{\alpha=1,2}h^{-1}_e\int_{\t e}\lbra u\e_{\alpha}-u^I_\alpha\rbra \lbra v_{\alpha}\rbra+h^{-3}_e\int_{e}\lbra w\e-w^I\rbra \lbra z\rbra
+\sum_{\alpha=1,2}h^{-1}_e\int_{e}\lbra \partial_\alpha w\e-\partial_\alpha w^I\rbra \lbra \partial_\alpha z\rbra \right)\right|\\
\le C
\left[
\sum_{\tau\in\T_h}\left(
\sum_{k=0}^4h^{2k-4}_\tau|w\e-w^I|^2_{k, \tau}+\sum_{k=0}^3\sum_{\alpha=1,2}h^{2k-2}_\tau|u\e_{\alpha}-u^I_{\alpha}|^2_{k, \tau}\right)\right]^{1/2}
\|\!v, z\|_{\H_h}.
\end{multline}
The boundary penalty term in $a(\!u\e-\!u^I, w\e-w^I;\!v, z)$  is estimated in the same manner.
All these together proves the estimate  \eqref{a-error}.
\end{proof}

\begin{lem}\label{b2-error-lem}
There is a $C$ independent of $\T_h$ such that
\begin{multline}\label{b2-error}
\left|b(\M\e-\M^I; \!v,z)\right|\\
\le C
\left[\sum_{\tau\in\T_h}\sum_{\alpha,\beta=1,2}\left(|\M^{\eps\alpha\beta}-\M^{I \alpha\beta}|^2_{0, \tau}+h^2_{\tau}|\M^{\eps\alpha\beta}-\M^{I\alpha\beta}|^2_{1, \tau}\right)
\right]^{1/2}\|\!v, z\|_{\H_h}\ \forall\ (\!v, z)\in\H_h.
\end{multline}
\end{lem}
\begin{proof}
In view of the definition \eqref{form_b}, we have
\begin{multline*}
b(\M\e-\M^I; \!v,z)=\\
\int_{\t\Omega_h}(\M^{\eps\alpha\beta}-\M^{I\alpha\beta})\gamma_{\alpha\beta}(\!v, z)
-\int_{\t\E^0_h}\lbrac\M^{\eps\alpha\beta}-\M^{I\alpha\beta}\rbrac\lbra v_{\alpha}\rbra_{n_{\beta}}
 -\int_{\t\E^{S}_h\cup\t\E^D_h}
(\M^{\eps\alpha\beta}-\M^{I\alpha\beta}){n_{\beta}}v_{\alpha}.
\end{multline*}
We have the estimates that
\begin{equation*}
\left|\int_{\t\Omega_h}(\M^{\eps\alpha\beta}-\M^{I\alpha\beta})\gamma_{\alpha\beta}(\!v, z)\right|
\le C
\sum_{\alpha, \beta=1,2}|\M^{\eps\alpha\beta}-\M^{I\alpha\beta}|_{0, \Omega_h}\sum_{\alpha, \beta=1,2}|\gamma_{\alpha\beta}(\!v, z)|_{0, \Omega_h},
\end{equation*}
\begin{multline*}
\left|\int_{\t\E^0_h}\lbrac\M^{\eps\alpha\beta}-\M^{I\alpha\beta}\rbrac\lbra v_{\alpha}\rbra_{n_{\beta}}\right|\\
\le C
\sum_{e\in\E^0_h}\left[\sum_{\alpha, \beta=1,2}h_e|\M^{\eps\alpha\beta}-\M^{I \alpha\beta}|^2_{0, e}\right]^{1/2}
\left[h^{-1}_e|\lbra \!v\rbra|^2_{0,e}\right]^{1/2}\\
\le C
\left[\sum_{\tau\in\T_h}\sum_{\alpha,\beta=1,2}\left(|\M^{\eps\alpha\beta}-\M^{I \alpha\beta}|^2_{0, \tau}+h^2_{\tau}|\M^{\eps\alpha\beta}-\M^{I\alpha\beta}|^2_{1, \tau}\right)
\right]^{1/2}\left[\sum_{e\in\E^0_h}h^{-1}_e|\lbra \!v\rbra|^2_{0,e}\right]^{1/2},
\end{multline*}
and
\begin{multline*}
\left|\int_{\t\E^{S}_h\cup\t\E^D_h}
(\M^{\eps\alpha\beta}-\M^{I\alpha\beta}){n_{\beta}}v_{\alpha}\right|\\
\le C
\sum_{e\in\E^S_h\cup\E^D_h}\left[\sum_{\alpha, \beta=1,2}h_e|\M^{\eps\alpha\beta}-\M^{I \alpha\beta}|^2_{0, e}\right]^{1/2}
\left[h^{-1}_e|\!v|^2_{0,e}\right]^{1/2}\\
\le C
\left[\sum_{\tau\in\T_h}\sum_{\alpha,\beta=1,2}\left(|\M^{\eps\alpha\beta}-\M^{I \alpha\beta}|^2_{0, \tau}+h^2_{\tau}|\M^{\eps\alpha\beta}-\M^{I\alpha\beta}|^2_{1, \tau}\right)
\right]^{1/2}\left[\sum_{e\in\E^S_h\cup\E^D_h}h^{-1}_e|\!v|^2_{0,e}\right]^{1/2}.
\end{multline*}
Sum up, we get the estimate \eqref{b2-error}.
\end{proof}
In both the inequalities \eqref{a-error} and \eqref{b2-error}, we did not
impose any condition for the interpolations $\!u^I$, $w^I$, and $\M^I$, except that they
are finite element functions from the space \eqref{FE-space}. The next estimate is very different in that
the interpolation needs to be  particularly chosen to obtain a desirable bound for
\begin{multline*}
b(\N; \!u\e-\!u^I, w\e-w^I)=
\int_{\t\Omega_h}\N^{\alpha\beta}\gamma_{\alpha\beta}(\!u\e-\!u^I, w\e-w^I)\\
-\int_{\t\E^0_h}\lbrac\N^{\alpha\beta}\rbrac\lbra u\e_{\alpha}-u^I_\alpha\rbra_{n_{\beta}}
 -\int_{\t\E^{S}_h\cup\t\E^D_h}
\N^{\alpha\beta}{n_{\beta}}(u\e_{\alpha}-u^I_\alpha).
\end{multline*}

On a $\tau\in\T_h$, we define $w^I\in P^3(\tau)$ by
\begin{equation}\label{wI}
\int_{\t\tau}(w\e-w^I)p=\int_{\tau}(w\e-w^I)p\sqrt a=0\ \forall\ p\in P^3(\tau).
\end{equation}
If $\partial\tau\cap\E^F_h=\emptyset$, we define $u^I_\alpha\in P^2(\tau)$ by
\begin{equation}\label{uI-interior}
\int_{\t\tau}(u\e_\alpha-u^I_\alpha)p=0\ \forall\ p\in P^2(\tau).
\end{equation}
If $\partial\tau\cap\E^F_h$ has one edge $e$, we define $u^I_\alpha\in P^3_*(\tau)$ by
\begin{equation}\label{uI-edge1}
\int_{\t\tau}(u\e_\alpha-u^I_\alpha)p=0\ \forall\ p\in P^2(\tau), \quad
\int_e(u\e_\alpha-u^I_\alpha) p\sqrt a=0\ \forall\ p\in P^1(e).
\end{equation}
If $\partial\tau\cap\E^F_h$ has two edges $e_\beta$, we define $u^I_\alpha\in P^3(\tau)$ by
\begin{equation}\label{uI-edge2}
\int_{\t\tau}(u\e_\alpha-u^I_\alpha)p=0\ \forall\ p\in P^2(\tau), \quad
\int_{e_\beta}(u\e_\alpha-u^I_\alpha) p\sqrt a=0\ \forall\ p\in P^1(e_\beta).
\end{equation}
The unisolvences of \eqref{wI} and \eqref{uI-interior} are trivial. The unisolvence of \eqref{uI-edge1} is seen from the condition
\eqref{P3*}. To see the unisolvence of \eqref{uI-edge2}, one may use Appell's polynomial \cite{Braess}
to decompose
a cubic polynomial as the sum of a quadratic and an orthogonal complement, and to see that
the orthogonal complement is uniquely determined by its averages and first moments on two edges.
\begin{lem}\label{b3-error-lem}
With the  interpolations defined above, there is a constant $C$ independent of $\T_h$ such that
\begin{multline}\label{b3-error}
|b(\N; \!u\e-\!u^I,w\e-w^I)|\le C
\max_{\tau\in\T_h}\left[h^{3}_\tau\sum_{\alpha,\beta,\lambda=1,2}|\Gamma^{\lambda}_{\alpha\beta}|_{2,\infty,\tau}+
h^{5}_\tau\sum_{\alpha,\beta=1,2}|b_{\alpha\beta}|_{3,\infty,\tau}\right]\\
\|\N\|_{\V_h}
\left[\sum_{\tau\in\T_h}h^{-2}_{\tau}\left|u\e_{\alpha}-u^I_{\alpha}\right|^2_{0,\tau}+
h^{-4}_{\tau}\left|w\e-w^I\right|^2_{0,\tau}\right]^{1/2}\ \ \forall\ \N\in\V_h.
\end{multline}
\end{lem}
\begin{proof}
With an application of the Green's theorem \eqref{Green} on each element $\t\tau\in\t\T_h$, summing up, and using the fact that $\N$ is continuous
on $\E^0_h$, we obtain the following  alternative expression.
\begin{multline*}
b(\N; \!u\e-\!u^I,w\e-w^I)=
\int_{\t\Omega_h}\left[-\N^{\alpha\beta}|_{\beta}\left(u\e_{\alpha}-u^I_{\alpha}\right)-b_{\alpha\beta}\N^{\alpha\beta}
\left(w\e-w^I\right)\right]\\
+\int_{\t\E_h^0}\lbra\N^{\alpha\beta}\rbra_{n_{\beta}}\lbrac
\left(u\e_{\alpha}-u^I_{\alpha}\right)\rbrac+
\int_{\t\E_h^F}\N^{\alpha\beta}n_{\beta}
\left(u\e_{\alpha}-u^I_{\alpha}\right)
\\
=
\int_{\t\Omega_h}\left[
-\N^{\alpha\beta}|_{\beta}
\left(u\e_{\alpha}-u^I_{\alpha}\right)-b_{\alpha\beta}\N^{\alpha\beta}\left(w\e-w^I\right)\right]
+\int_{\t\E_h^F}\N^{\alpha\beta}n_{\beta}
\left(u\e_{\alpha}-u^I_{\alpha}\right).
\end{multline*}
For each $e\in\E^F_h$, we have, see \eqref{Green} in the appendix,
\begin{equation*}
\int_{\t e}\N^{\alpha\beta}n_{\beta}
\left(u\e_{\alpha}-u^I_{\alpha}\right)=\int_{e}\N^{\alpha\beta}\bar n_{\beta}
\left(u\e_{\alpha}-u^I_{\alpha}\right)\sqrt a=0.
\end{equation*}
We  thus have
\begin{multline}\label{b-P2P3}
b(\N; \!u\e-\!u^I,w\e-w^I)\\=
-\sum_{\tau\in\T_h}\int_{\t\tau}\left\{
\left[
\N^{\alpha\beta}|_{\beta}
-P_2(\N^{\alpha\beta}|_{\beta})
\right]
\left(u\e_{\alpha}-u^I_{\alpha}\right)+
\left[b_{\alpha\beta}\N^{\alpha\beta}-
P_3(b_{\alpha\beta}\N^{\alpha\beta})\right]
\left(w\e-w^I\right)\right\}.
\end{multline}
Here $P_2(\N^{\alpha\beta}|_{\beta})$ is the $L^2$ projection (weighted by $\sqrt a$) of $\N^{\alpha\beta}|_{\beta}$
into the space of quadratic polynomials on the element $\tau$, and $P_3(b_{\alpha\beta}\N^{\alpha\beta})$
is the weighted $L^2$ projection of  $b_{\alpha\beta}\N^{\alpha\beta}$ into the space of cubic polynomials on $\tau$.
On an element $\t\tau$, we have
\begin{multline}\label{b-on-tau}
\left|\int_{\t\tau}
\left[
\N^{\alpha\beta}|_{\beta}
-P_2(\N^{\alpha\beta}|_{\beta})
\right]
\left(u\e_{\alpha}-u^I_{\alpha}\right)\right|\\
\le C
\left[h_\tau\sum_{\alpha=1,2}\left|\N^{\alpha\beta}|_{\beta}
-P_2(\N^{\alpha\beta}|_{\beta})\right|_{0,\tau}\right]
h^{-1}_\tau|\!u\e-\!u^I|_{0,\tau}.
\end{multline}
In view of the formula
$ 
\N^{\alpha\beta}|_{\beta}=\partial_\beta\N^{\alpha\beta}+\Gamma^{\beta}_{\beta\gamma}\N^{\alpha\gamma}+
\Gamma^{\alpha}_{\delta\beta}\N^{\delta\beta}
$ 
we have 
\begin{equation*}
\left|\N^{\alpha\beta}|_{\beta}
-P_2(\N^{\alpha\beta}|_{\beta})\right|_{0,\tau}\le
\left|\Gamma^{\beta}_{\beta\gamma}-p_1(\Gamma^{\beta}_{\beta\gamma})\right|_{0,\infty,\tau}
\left|\N^{\alpha\gamma}\right|_{0,\tau}+
\left|\Gamma^{\alpha}_{\beta\gamma}-p_1(\Gamma^{\alpha}_{\beta\gamma})\right|_{0,\infty,\tau}
\left|\N^{\beta\gamma}\right|_{0,\tau}.
\end{equation*}
Here, $p_1(\Gamma^{\beta}_{\beta\gamma})$ is the best linear approximation to $\Gamma^{\beta}_{\beta\gamma}$
in the space $L^{\infty}(\tau)$ such that
\begin{equation*}
\left|\Gamma^{\alpha}_{\beta\gamma}-p_1(\Gamma^{\beta}_{\beta\gamma})\right|_{0,\infty,\tau}\le Ch^2_\tau \left|\Gamma^{\alpha}_{\beta\gamma}\right|_{2,\infty,\tau}.
\end{equation*}
Summing \eqref{b-on-tau} for all $\tau\in\T_h$, and using Cauchy--Schwarz inequality, we get
\begin{multline*}
\sum_{\tau\in\T_h}\int_{\t\tau}
\left|\left[
\N^{\alpha\beta}|_{\beta}
-P_2(\N^{\alpha\beta}|_{\beta})
\right]
\left(u\e_{\alpha}-u^I_{\alpha}\right)\right|\\
\le C
\left[\sum_{\tau\in\T_h}\left(h^6_\tau\sum_{\alpha,\beta,\gamma=1,2}
\left|\Gamma^{\alpha}_{\beta\gamma}\right|^2_{2,\infty,\tau}\sum_{\alpha,\beta=1,2}|\N^{\alpha\beta}|^2_{0,\tau}\right)
\right]^{1/2}
\left[\sum_{\tau\in\T_h}h^{-2}_\tau|\!u\e-\!u^I|^2_{0,\tau}
\right]^{1/2}\\
\le C
\max_{\tau\in\T_h}\left(h^{3}_\tau\sum_{\alpha,\beta,\lambda=1,2}|\Gamma^{\lambda}_{\alpha\beta}|_{2,\infty,\tau}\right)\|\N\|_{\V_h}
\left[\sum_{\tau\in\T_h}h^{-2}_\tau|\!u\e-\!u^I|^2_{0,\tau}
\right]^{1/2}.
\end{multline*}
Similarly,
\begin{multline*}
\sum_{\tau\in\T_h}\left|\int_{\t\tau}
\left[b_{\alpha\beta}\N^{\alpha\beta}-
P_3(b_{\alpha\beta}\N^{\alpha\beta})\right]
\left(w\e-w^I\right)\right|\\
\le C
\max_{\tau\in\T_h}\left(h^{5}_\tau\sum_{\alpha,\beta=1,2}|b_{\alpha\beta}|_{3,\infty,\tau}\right)\|\N\|_{\V_h}
\left[\sum_{\tau\in\T_h}h^{-4}_\tau|w\e-w^I|^2_{0,\tau}
\right]^{1/2}.
\end{multline*}
\end{proof}
It is trivial to see that
\begin{multline}\label{c-error}
\left|c(\M\e-\M^I, \N)\right|
\le C \|\N\|_{\V_h}
\sum_{\alpha,\beta=1,2}|\M^{\eps\alpha\beta}-\M^{I\alpha\beta}|_{0,\Omega_h}\ \forall\ \N\in\V_h.\hfill
\end{multline}

The following estimate is a result of combining  \eqref{a-error}, \eqref{b2-error}, \eqref{b3-error}, \eqref{c-error}, and \eqref {error-fraction}.
\begin{thm}\label{K-fem-lem}
Let $u^I_{\alpha}$ and $w^I$ be the approximations to $u\e_\alpha$ and $w\e$
in the finite element space \eqref{FE-space}, which is defined by the formulas \eqref{wI}, \eqref{uI-interior},
\eqref{uI-edge1}, and \eqref{uI-edge2}. Let
$\M^{I\alpha\beta}$ be an approximation
to $\M^{\eps\alpha\beta}$ from the space of continuous piecewise linear functions. There is a $C$
independent of $\T_h$ such that
\begin{multline}\label{K-fem-error-lem}
\|(\!u^h-\!u^I, w^h-w^I)\|_{\H_h}+
\|\M^h-\M^I\|_{\overline\V_h}+\eps\|\M^h-\M^I\|_{\V_h}
\\
\le C
\left[1+\eps^{-1}\max_{\tau\in\T_h}\left(h^{3}_\tau\sum_{\alpha,\beta,\lambda=1,2}|\Gamma^{\lambda}_{\alpha\beta}|_{2,\infty,\tau}+
h^{5}_\tau\sum_{\alpha,\beta=1,2}|b_{\alpha\beta}|_{3,\infty,\tau}\right)
\right]
\\
\left[
\sum_{\tau\in\T_h}\left(
\sum_{k=0}^4h^{2k-4}_\tau|w\e-w^I|^2_{k, \tau}+\sum_{k=0}^3\sum_{\alpha=1,2}h^{2k-2}_\tau|u\e_{\alpha}-u^I_{\alpha}|^2_{k, \tau}\right.\right.\\
\left.\left.+\sum_{k=0}^1\sum_{\alpha,\beta=1,2}h^{2k}_\tau|\M^{\eps\alpha\beta}-\M^{I\alpha\beta}|^2_{k, \tau}\right)\right]^{1/2}.
\end{multline}
\end{thm}

We have the following theorem on the error estimate for the finite element method introduced this paper.
\begin{thm}
If the Koiter model solution has the regularity that $u\e_\alpha\in H^3$ and $w\e\in H^4$, then
there is a constant $C$ that is independent of the triangulation $\T_h$ and the shell thickness $\eps$, such that
\begin{multline*}
\|(\!u\e-\!u^h, w\e-w^h)\|_{\H_h}\\
\le C
\left[1+\eps^{-1}\max_{\tau\in\T_h}\left(h^{3}_\tau\sum_{\alpha,\beta,\lambda=1,2}|\Gamma^{\lambda}_{\alpha\beta}|_{2,\infty,\tau}+
h^{5}_\tau\sum_{\alpha,\beta=1,2}|b_{\alpha\beta}|_{3,\infty,\tau}\right)
\right]\\
\left[
\sum_{\tau\in\T_h}h^4_{\tau}\left(\sum_{\alpha=1,2}\|u\e_\alpha\|^2_{3,\tau}+\|w\e\|^2_{4,\tau}+
\sum_{\alpha, \beta=1,2}\|\M^{\eps\alpha\beta}\|^2_{2,\tau}
\right)\right]^{1/2}.
\end{multline*}
Here $\!u^h$, $w^h$, and $\M^h$ is the solution of the
the finite element model \eqref{K-fem} with the finite element space defined by \eqref{FE-space}. The norm $\|\cdot\|_{\H_h}$
is defined by \eqref{Hh-norm}.
\end{thm}
\begin{proof}
In view of the triangle inequality, we have
\begin{equation*}
\|(\!u\e-\!u^h, w\e-w^h)\|_{\H_h}\le \|(\!u\e-\!u^I, w\e-w^I)\|_{\H_h}+\|(\!u^h-\!u^I, w^h-w^I)\|_{\H_h}.
\end{equation*}
Using the trace inequality \eqref{trace} to the edge terms in the norm  $\|(\!u\e-\!u^I, w\e-w^I)\|_{\H_h}$, cf., \eqref{Hh-norm}, we get
\begin{equation*}
\|(\!u\e-\!u^I, w\e-w^I)\|_{\H_h}\le C
\left[
\sum_{\tau\in\T_h}\left(
\sum_{k=0}^2h^{2k-4}_\tau|w\e-w^I|^2_{k, \tau}+\sum_{k=0}^1\sum_{\alpha=1,2}h^{2k-2}_\tau|u\e_{\alpha}-u^I_{\alpha}|^2_{k, \tau}
\right)\right]^{1/2}.
\end{equation*}
For each $\tau\in\T_h$, we need to establish that
\begin{equation}\label{scaled-estimate}
\begin{gathered}
\sum_{k=0}^4h^{2k-4}_\tau|w\e-w^I|^2_{k, \tau}\le C h^4_\tau|w\e|^2_{4,\tau},\\
\sum_{k=0}^3\sum_{\alpha=1,2}h^{2k-2}_\tau|u\e_{\alpha}-u^I_{\alpha}|^2_{k, \tau}\le C h^4_\tau|\!u\e|_{3,\tau}.
\end{gathered}
\end{equation}

We scale $\tau$ to a similar triangle $\T$ whose diameter is $1$ by the scaling $X_\alpha=h^{-1}_\tau x_\alpha$. Let $W(X_\alpha)=w(x_\alpha)$,
$U_\beta(X_\alpha)=u_\beta(x_\alpha)$
$A(X_\alpha)=a(x_\alpha)$, $W^I(X_\alpha)=w^I(x_\alpha)$, and $U^I_\beta(X_\alpha)=u^I_\beta(x_\alpha)$. It is easy to see that $W^I$ is the projection
of $W$ into $P^3(\T)$ in the space $L^2(\T)$ weighted by $\sqrt{A(X_\alpha)}$.
This projection preserves cubic polynomials and we have the bound that
\begin{equation}\label{WI-bound}
\|W^I\|_{0,T}\le \left[\frac{\max_{\tau}a}{\min_{\tau}a}\right]^{1/4}\|W\|_{0,\T}\ \ \forall\ W\in L^2(\T).
\end{equation}
For a $W\in H^4(\T)$ and any cubic polynomial $p$, using inverse inequality, there is a $C$ depending on the shape regularity of $\T$
such that
\begin{equation*}
\|W-W^I\|_{4, \T}\le \|W-p\|_{4, \T}+\|(W-p)^I\|_{4, \T}\le \|W-p\|_{4, \T}+C\|(W-p)^I\|_{0, \T}.
\end{equation*}
Therefore, there is a $C$ depending on the shape regularity of $\T$ and the  ratio ${\max_{\tau}a}/{\min_{\tau}a}$ such that
\begin{equation*}
\|W-W^I\|_{4, \T}\le C\|W-p\|_{4, \T}\ \ \forall\ p\in P^3(\T).
\end{equation*}
Using the interpolation operator of \cite{Verfurth}, we can choose a $p\in P^3(\T)$ and an absolute constant such that
 \begin{equation*}
\|W-W^I\|_{4, \T}\le C\|W-p\|_{4, \T}\le C |W|_{4, \T}.
\end{equation*}
Scale this estimate from $\T$ to $\tau$, we obtain the first estimate in \eqref{scaled-estimate}.

If $\tau$ has no edge on the free boundary $\E^F_h$, the second inequality in \eqref{scaled-estimate} is proved
in the same way except that the $P^3(\T)$ is replaced by $P^2(\T)$.
If $\tau$ has one or two edges on the free boundary, in place of the estimate
\eqref{WI-bound}, we have that there is a $C$ depending only on the shape regularity of $\T$ such that
\begin{equation*} 
\|U^I_\alpha\|_{0,T}\le C \|U_\alpha\|_{1,\T}\ \ \forall\ U_\alpha \in H^1(\T).
\end{equation*}
For any $p\in P^2(\T)$, we have
\begin{multline*}
\|U_\alpha-U^I_\alpha\|_{3, \T}\le \|U_\alpha-p\|_{3, \T}+\|(U_\alpha-p)^I\|_{3, \T}\\
\le \|U_\alpha-p\|_{3, \T}+C\|(U_\alpha-p)^I\|_{0, \T}
\le \|U_\alpha-p\|_{3, \T}+C\|U_\alpha-p\|_{1, \T}
\le C\|U_\alpha-p\|_{3, \T}.
\end{multline*}
Using the interpolation operator of \cite{Verfurth} again, we get
\begin{equation*}
\|U_\alpha-U^I_\alpha\|_{3, \T}\le C|U_\alpha|_{3,\T}.
\end{equation*}
Here $C$ only depends on the shape regularity of $\T$.
The second inequality in \eqref{scaled-estimate} then follows the scaling from $\T$ to $\tau$.

Finally, we need to show that there is an interpolation $M^{I\alpha\beta}$ from continuous piecewise linear functions for $\M^{\eps\alpha\beta}$
such that
\begin{equation*}
\sum_{\tau\in\T_h}\left(|\M^{\eps\alpha\beta}-\M^{I\alpha\beta}|^2_{0, \tau}+h^2_{\tau}|\M^{\eps\alpha\beta}-\M^{I\alpha\beta}|^2_{1, \tau}\right)
\le C
\sum_{\tau\in\T_h}h^4_{\tau}\|\M^{\eps\alpha\beta}\|^2_{2,\tau}.
\end{equation*}
We choose the interpolation $\M^{I\alpha\beta}$ as the $L^2$ projection in the linear continuous finite element space of $\M^{\eps\alpha\beta}$.
On a $\tau\in\T_h$, we let $\M^{I\alpha\beta}_{\tau}$ be a local $L^2$ projection (or interpolation)
of $\M^{\eps\alpha\beta}$. We then have
\begin{multline*}
h^2_{\tau}|\M^{I\alpha\beta}-\M^{\eps\alpha\beta}|^2_{1, \tau}
\le
h^2_{\tau}|\M^{I\alpha\beta}-\M^{I\alpha\beta}_{\tau}|^2_{1, \tau}+h^2_{\tau}|\M^{I\alpha\beta}_{\tau}-\M^{\eps\alpha\beta}|^2_{1, \tau}\\
\le
|\M^{I\alpha\beta}-\M^{I\alpha\beta}_{\tau}|^2_{0, \tau}+h^2_{\tau}|\M^{I\alpha\beta}_{\tau}-\M^{\eps\alpha\beta}|^2_{1, \tau}\\
\le
|\M^{I\alpha\beta}-\M^{\eps\alpha\beta}|^2_{0, \tau}+
|\M^{\eps\alpha\beta}-\M^{I\alpha\beta}_{\tau}|^2_{0, \tau}
+h^2_{\tau}|\M^{I\alpha\beta}_{\tau}-\M^{\eps\alpha\beta}|^2_{1, \tau}.
\end{multline*}
Here, we used the inverse inequality for finite element functions, which is valid since our triangulation is shape regular.
We thus proved that
\begin{multline*}
\sum_{\tau\in\T_h}
\left[|\M^{\eps\alpha\beta}-\M^{I\alpha\beta}|^2_{0, \tau}+h^2_{\tau}|\M^{\eps\alpha\beta}-\M^{I\alpha\beta}|^2_{1, \tau}\right]\\
\le
\sum_{\tau\in\T_h}|\M^{\eps\alpha\beta}-\M^{I\alpha\beta}|^2_{0, \tau}
+
\sum_{\tau\in\T_h}|\M^{\eps\alpha\beta}-\M^{I\alpha\beta}_{\tau}|^2_{0, \tau}
+
\sum_{\tau\in\T_h}
h^2_{\tau}|\M^{\eps\alpha\beta}-\M^{I\alpha\beta}_{\tau}|^2_{1, \tau}\\
\le
\sum_{\tau\in\T_h}
\left[|\M^{\eps\alpha\beta}-\M^{I\alpha\beta}_{\tau}|^2_{0, \tau}+h^2_{\tau}|\M^{\eps\alpha\beta}-\M^{I\alpha\beta}_{\tau}|^2_{1, \tau}\right]
\le C
\sum_{\tau\in\T_h}h^4_{\tau}\|\M^{\eps\alpha\beta}\|^2_{2,\tau}.
\end{multline*}
The desired estimate then follows from Theorem~\ref{K-fem-lem}.
\end{proof}


\bibliographystyle{plain}

\appendix
\section*{Appendix: Consistency of the finite element model}
\renewcommand{\theequation}{A.\arabic{equation}}
\setcounter{equation}{0}
We verify that the solution $\!u\e, w\e, \M\e$ of the Koiter model \eqref{K-P-model} satisfies
the equation of the finite element model \eqref{K-fem} in which the test function $\!v, z, \N$ can be any piecewise
functions of sufficient regularity, not necessarily polynomials.
For this purpose,
we need to repeatedly do integration by parts on the shell mid-surface, by using the Green's theorem on surfaces.
Let $\tau\subset\Omega$ be a subdomain, which is mapped to
the subregion $\tilde\tau\subset\t\Omega$ by $\!phi$.
Let $\!n=n_{\alpha}\!a^{\alpha}=n^\alpha\!a_\alpha$
be the unit outward normal to the boundary $\partial\tilde\tau=\!phi(\partial\tau)$
which is tangent to the surface $\t\Omega$.
Let $\bar n_\alpha\!e^\alpha$ be the  unit outward normal vector to $\partial\tau$ in $\R^2$. Here $\!e^\alpha$ is the basis vector in $\R^2$.
The Green's theorem says that for a vector field $f^\alpha$,
\begin{equation}\label{Green}
\int_{\tilde\tau}f^{\alpha}|_{\alpha}=
\int_{\partial\tilde\tau}f^{\alpha}n_{\alpha}=
\int_{\partial\tau}f^{\alpha}\bar n_{\alpha}\sqrt a.
\end{equation}

Under the assumption that  the shell material has constant
Lam\'e coefficients, we have $a^{\alpha\beta\gamma\delta}|_{\tau}=a_{\alpha\beta\gamma\delta}|_{\tau}=0$.
This is due to the fact that $a^{\alpha\beta}|_{\gamma}=a_{\alpha\beta}|_{\gamma}=0$.
On a region  $\t\tau\subset\t\Omega$, for any vectors $u_\alpha$ and $v_\alpha$, scalars $w$ and $z$, and symmetric tensor $\M^{\alpha\beta}$,
the following identities follow from the Green's theorem and the definitions of change of curvature \eqref{K-curvature} and
change of metric \eqref{K-metric} tensors.
\begin{multline}\label{id-1}
\int_{\t\tau}
a^{\alpha\beta\lambda\gamma}\rho_{\lambda\gamma}(\!u, w)(\partial_\alpha z)|_{\beta}=
\int_{\t\tau}
a^{\alpha\beta\lambda\gamma}\rho_{\lambda\gamma|\alpha\beta}(\!u, w)
z\\+
\int_{\partial\t\tau}
a^{\alpha\beta\lambda\gamma}\rho_{\lambda\gamma}(\!u, w){n_{\beta}}
\partial_\alpha z-
\int_{\partial\t\tau}a^{\alpha\beta\lambda\gamma}\rho_{\lambda\gamma|\beta}(\!u, w)n_{\alpha}z.
\end{multline}
\begin{multline*}
\int_{\t\tau}
a^{\gamma\beta\sigma\tau}\rho_{\sigma\tau}(\!u, w)
(b^{\alpha}_{\beta}v_{\alpha})|_{\gamma}=
-\int_{\t\tau}
a^{\gamma\beta\sigma\tau}\rho_{\sigma\tau|\gamma}(\!u, w)
b^{\alpha}_{\beta}v_{\alpha}+
\int_{\partial\t\tau}a^{\gamma\beta\sigma\tau}\rho_{\sigma\tau}(\!u, w)b^{\alpha}_{\beta}
v_{\alpha}n_{\gamma}.\hfill
\end{multline*}
\begin{multline*}
\int_{\t\tau}
a^{\gamma\beta\sigma\tau}\rho_{\sigma\tau}(\!u, w)
b^{\alpha}_{\gamma}v_{\alpha|\beta}
=
-\int_{\t\tau}
a^{\gamma\beta\sigma\tau}[\rho_{\sigma\tau}(\!u, w)
b^{\alpha}_{\gamma}]|_{\beta}v_{\alpha}
+\int_{\partial\t\tau}a^{\gamma\beta\sigma\tau}\rho_{\sigma\tau}(\!u, w){n_{\beta}}
b^{\alpha}_{\gamma}v_{\alpha}.\hfill
\end{multline*}
\begin{multline*}
\int_{\t\tau}a^{\alpha\beta\gamma\delta}\gamma_{\gamma\delta}(\!u, w)\frac12(v_{\alpha|\beta}+v_{\beta|\alpha})=
-\int_{\t\tau}a^{\alpha\beta\gamma\delta}\gamma_{\gamma\delta|\beta}(\!u, w)v_{\alpha}
+\int_{\partial\t\tau}a^{\alpha\beta\gamma\delta}\gamma_{\gamma\delta}(\!u, w){n_{\beta}}v_{\alpha}.\hfill
\end{multline*}
\begin{multline*}
\int_{\t\tau}\M^{\alpha\beta}\frac12(v_{\alpha|\beta}+v_{\beta|\alpha})=
-\int_{\t\tau}\M^{\alpha\beta}|_{\beta}v_{\alpha}
+\int_{\partial\t\tau}\M^{\alpha\beta}{n_{\beta}}v_{\alpha}.\hfill
\end{multline*}

On the boundary $\partial\t\tau$, in addition to the normal vector $\!n$,
we introduce the  counterclockwise unit tangent vector $\!s=s^{\alpha}\!a_{\alpha}=s_\alpha\!a^\alpha$. Then the second integrand in the right hand side
of the above first equation can be further written as
\begin{multline*}
a^{\alpha\beta\lambda\gamma}\rho_{\lambda\gamma}(\!u, w){n_{\beta}}
\partial_\alpha z=
a^{\alpha\beta\lambda\gamma}\rho_{\lambda\gamma}(\!u, w){n_{\beta}}n_\alpha
n^\delta\partial_\delta z+
a^{\alpha\beta\lambda\gamma}\rho_{\lambda\gamma}(\!u, w){n_{\beta}}s_\alpha
s^\delta\partial_\delta z\\
=
a^{\alpha\beta\lambda\gamma}\rho_{\lambda\gamma}(\!u, w){n_{\beta}}n_\alpha
D_{\!n}z+
a^{\alpha\beta\lambda\gamma}\rho_{\lambda\gamma}(\!u, w){n_{\beta}}s_\alpha
D_{\!s}z.
\end{multline*}
Here $z$ is viewed as a function defined on the surface $\t\Omega$, and the invariant
$D_{\!s}z=s^\alpha\partial_\alpha z$ is the directional derivative in the direction of $\!s$, with respect to arc length.
The invariant $D_{\!n}z=n^\alpha\partial_\alpha z$ is the directional derivative of $z$ in the normal direction, with respect to arc length.

Using the Green's theorem on $\t\Omega$ several times,
we  write the Koiter model \eqref{K-P-model} in the following mixed strong form.
\begin{equation}\label{K-P-strong}
\begin{gathered}
\hfill -\frac13\left\{a^{\gamma\beta\sigma\tau}\rho_{\sigma\tau|\gamma}(\!u, w)
b^{\alpha}_{\beta}+
a^{\gamma\beta\sigma\tau}[\rho_{\sigma\tau}(\!u, w)
b^{\alpha}_\gamma]|_{\beta}
+a^{\alpha\beta\gamma\delta}\gamma_{\gamma\delta|\beta}(\!u, w)\right\}
-
\M^{\alpha\beta}|_{\beta}=p^{\alpha}\text{ in }\Omega,\\
\hfill \frac13\left\{a^{\alpha\beta\lambda\gamma}\rho_{\lambda\gamma|\alpha\beta}(\!u, w)
-c_{\alpha\beta}
a^{\alpha\beta\lambda\gamma}\rho_{\lambda\gamma}(\!u, w)
-b_{\alpha\beta}a^{\alpha\beta\lambda\gamma}\gamma_{\lambda\gamma}(\!u, w)\right\}
-b_{\alpha\beta}\M^{\alpha\beta}=p^3 \text{ in }\Omega, \\
\hfill \gamma_{\alpha\beta}(\!u, w)
-\eps^2a_{\alpha\beta\lambda\gamma}\M^{\lambda\gamma}=0 \text{ in }\Omega,
\end{gathered}
\end{equation}
\begin{equation*}
\begin{gathered}
\hfill  2\frac13a^{\gamma\beta\sigma\tau}\rho_{\sigma\tau}(\!u, w)b^{\alpha}_{\beta}
{n_{\gamma}}
+\frac13a^{\alpha\beta\gamma\delta}\gamma_{\gamma\delta}(\!u, w){n_{\beta}}
+\M^{\alpha\beta}{n_{\beta}}=q^\alpha\text{ on }    \partial^{F}\Omega,\\
\hfill
-\frac13a^{\alpha\beta\lambda\gamma}\rho_{\lambda\gamma|\beta}(\!u, w)
{n_{\alpha}}-\frac13D_{\!s}\left[a^{\alpha\beta\lambda\gamma}\rho_{\lambda\gamma}(\!u, w)
n_{\beta}s_\alpha\right]=q^3\text{ on }    \partial^{F}\Omega,
\\
\hfill  \frac13a^{\alpha\beta\lambda\gamma}\rho_{\lambda\gamma}(\!u, w)
n_{\beta}n_\alpha =m\text{ on }    \partial^{F\cup S}\Omega,\\
\hfill w=u_{\alpha}=0 \text{ on }\partial^{S\cup D}\Omega,
\\
\hfill D_{\!n}w=0\text{ on }\partial^{D}\Omega.
\end{gathered}
\end{equation*}

For any piecewise vectors $u_\alpha$ and $v_\alpha$, scalars $w$ and $z$, and symmetric tensors
$\M^{\alpha\beta}$ and $\N^{\alpha\beta}$, on $\Omega_h$, summing up the bilinear forms
defined by \eqref{form_ub_a}, \eqref{form_b}, and \eqref{form_c}, we have

\begin{multline*}
\ub a(\!u, w;\!v, z)+b(\M; \!v, z)-b(\N; \!u, w)+\eps^2c(\M, \N)\hfill \\
=
\frac13\int_{\t\Omega_h}\left[a^{\alpha\beta\lambda\gamma}
\rho_{\lambda\gamma}(\!u, w)\rho_{\alpha\beta}(\!v, z)
+a^{\alpha\beta\lambda\gamma}
\gamma_{\lambda\gamma}(\!u, w)\gamma_{\alpha\beta}(\!v, z)\right]+\int_{\t\Omega_h}\M^{\alpha\beta}\gamma_{\alpha\beta}(\!v, z)
\\
\hfill
-\int_{\t\Omega_h}\N^{\alpha\beta}\gamma_{\alpha\beta}(\!u, w)+\eps^2\int_{\t\Omega_h}a^{\alpha\beta\gamma\delta}\M_{\gamma\delta}\N_{\alpha\beta}
\end{multline*}
\begin{multline*}
-\frac23\int_{\t\E^0_h}a^{\gamma\beta\sigma\tau}\lbrac\rho_{\sigma\tau}(\!v, z)\rbrac b^{\alpha}_{\beta}
\lbra u_{\alpha}\rbra_{n_{\gamma}}
-\frac23\int_{\t\E^0_h}a^{\gamma\beta\sigma\tau}\lbrac\rho_{\sigma\tau}(\!u, w)\rbrac b^{\alpha}_{\beta}
\lbra v_{\alpha}\rbra_{n_{\gamma}}\\
-\frac13\int_{\t\E^0_h}
a^{\alpha\beta\lambda\gamma}\lbrac\rho_{\lambda\gamma}(\!v, z)\rbrac
\lbra \partial_\alpha w\rbra_{n_{\beta}}
-\frac13\int_{\t\E^0_h}
a^{\alpha\beta\lambda\gamma}\lbrac\rho_{\lambda\gamma}(\!u, w)\rbrac
\lbra \partial_\alpha z\rbra_{n_{\beta}}\\
+
\frac13\int_{\t\E^0_h}
\lbrac a^{\alpha\beta\lambda\gamma}\rho_{\lambda\gamma|\beta}(\!v, z)
\rbrac\lbra w\rbra_{n_{\alpha}}
+
\frac13\int_{\t\E^0_h}
\lbrac a^{\alpha\beta\lambda\gamma}\rho_{\lambda\gamma|\beta}(\!u, w)
\rbrac\lbra z\rbra_{n_{\alpha}}\\
-\frac13\int_{\t\E^0_h}a^{\alpha\beta\gamma\delta}\lbrac\gamma_{\gamma\delta}(\!v, z)\rbrac\lbra u_{\alpha}\rbra_{n_{\beta}}
-\frac13\int_{\t\E^0_h}a^{\alpha\beta\gamma\delta}\lbrac\gamma_{\gamma\delta}(\!u, w)\rbrac\lbra v_{\alpha}\rbra_{n_{\beta}}\\
+\int_{\t\E^0_h}\lbrac\N^{\alpha\beta}\rbrac\lbra u_{\alpha}\rbra_{n_{\beta}}
-\int_{\t\E^0_h}\lbrac\M^{\alpha\beta}\rbrac\lbra v_{\alpha}\rbra_{n_{\beta}}
\end{multline*}
\begin{multline*}
-\int_{\t\E^S_h}\left[\frac23a^{\gamma\beta\sigma\tau}\rho_{\sigma\tau}(\!u, w)b^{\alpha}_{\beta}
{n_{\gamma}}
+\frac13a^{\alpha\beta\gamma\delta}\gamma_{\gamma\delta}(\!u, w){n_{\beta}}
+\M^{\alpha\beta}{n_{\beta}}\right]v_{\alpha}
\\
-\int_{\t\E^S_h}\left[\frac23a^{\gamma\beta\sigma\tau}\rho_{\sigma\tau}(\!v, z)b^{\alpha}_{\beta}
{n_{\gamma}}
+\frac13a^{\alpha\beta\gamma\delta}\gamma_{\gamma\delta}(\!v, z){n_{\beta}}
-\N^{\alpha\beta}{n_{\beta}}\right]u_{\alpha}
\\
+
\int_{\t\E^S_h}
\frac13\left\{a^{\alpha\beta\lambda\gamma}\rho_{\lambda\gamma|\beta}(\!u, w)
{n_{\alpha}}+ D_{\!s}\left[a^{\alpha\beta\lambda\gamma}\rho_{\lambda\gamma}(\!u, w)
n_{\beta}s_\alpha\right]\right\}z
\\
\hfill  +
\int_{\t\E^S_h}
\frac13\left\{a^{\alpha\beta\lambda\gamma}\rho_{\lambda\gamma|\beta}(\!v, z)
{n_{\alpha}}+ D_{\!s}\left[a^{\alpha\beta\lambda\gamma}\rho_{\lambda\gamma}(\!v, z)
n_{\beta}s_\alpha\right]\right\}w
\end{multline*}
\begin{multline*}
-\int_{\t\E^D_h}\left[\frac23a^{\gamma\beta\sigma\tau}\rho_{\sigma\tau}(\!u, w)b^{\alpha}_{\beta}
{n_{\gamma}}
+\frac13a^{\alpha\beta\gamma\delta}\gamma_{\gamma\delta}(\!u, w){n_{\beta}}
+\M^{\alpha\beta}{n_{\beta}}\right]v_{\alpha}
\\
 -\int_{\t\E^D_h}\left[\frac23a^{\gamma\beta\sigma\tau}\rho_{\sigma\tau}(\!v, z)b^{\alpha}_{\beta}
{n_{\gamma}}
+\frac13a^{\alpha\beta\gamma\delta}\gamma_{\gamma\delta}(\!v, z){n_{\beta}}
-\N^{\alpha\beta}{n_{\beta}}\right]u_{\alpha}
\end{multline*}
\begin{multline*}
+
\int_{\t\E^D_h}
\frac13\left\{a^{\alpha\beta\lambda\gamma}\rho_{\lambda\gamma|\beta}(\!u, w)
{n_{\alpha}}+ D_{\!s}\left[a^{\alpha\beta\lambda\gamma}\rho_{\lambda\gamma}(\!u, w)
n_{\beta}s_\alpha\right]\right\}z\\
+
\int_{\t\E^D_h}
\frac13\left\{a^{\alpha\beta\lambda\gamma}\rho_{\lambda\gamma|\beta}(\!v, z)
{n_{\alpha}}+ D_{\!s}\left[a^{\alpha\beta\lambda\gamma}\rho_{\lambda\gamma}(\!v, z)
n_{\beta}s_\alpha\right]\right\}w
\\
-\int_{\t\E^D_h}
\frac13a^{\alpha\beta\lambda\gamma}\rho_{\lambda\gamma}(\!u, w)
n_{\beta}n_\alpha D_{\!n}z
-\int_{\t\E^D_h}
\frac13a^{\alpha\beta\lambda\gamma}\rho_{\lambda\gamma}(\!v, u)
n_{\beta}n_\alpha D_{\!n}w.
\end{multline*}

Using the identities \eqref{id-1} on every element $\t\tau\in\t\T_h$, we can rewrite the above form as
\begin{multline}\label{K-ub-a-weak}
\ub a(\!u, w;\!v, z)+b(\M; \!v, z)-b(\N; \!u, w)+\eps^2c(\M, \N) \\
=\int_{\t\Omega_h}\left\{\frac13\left[-a^{\gamma\beta\sigma\tau}\rho_{\sigma\tau|\gamma}(\!u, w)
b^{\alpha}_{\beta}-
a^{\gamma\beta\sigma\tau}[\rho_{\sigma\tau}(\!u, w)
b^{\alpha}_{\gamma}]|_{\beta}
-a^{\alpha\beta\gamma\delta}\gamma_{\gamma\delta|\beta}(\!u, w)\right]
-
\M^{\alpha\beta}|_{\beta}\right\}v_{\alpha} \\
+\int_{\t\Omega_h}
\left\{\frac13\left[a^{\alpha\beta\lambda\gamma}\rho_{\lambda\gamma|\alpha\beta}(\!u, w)
-c_{\alpha\beta}
a^{\alpha\beta\lambda\gamma}\rho_{\lambda\gamma}(\!u, w)
-b_{\alpha\beta}a^{\alpha\beta\gamma\delta}\gamma_{\gamma\delta}(\!u, w)\right]
-b_{\alpha\beta}\M^{\alpha\beta}\right\}z  \\
-\int_{\t\Omega_h}\N^{\alpha\beta}\gamma_{\alpha\beta}(\!u, w)+\eps^2\int_{\t\Omega}a^{\alpha\beta\gamma\delta}\M_{\gamma\delta}\N_{\alpha\beta}
\end{multline}
\begin{multline*}
+\frac23\int_{\t\E^0_h}a^{\gamma\beta\sigma\tau}\lbra\rho_{\sigma\tau}(\!u, w)\rbra_{n_{\gamma}}
b^{\alpha}_{\beta}\lbrac v_{\alpha}\rbrac
+\frac13\int_{\t\E^0_h}
a^{\alpha\beta\lambda\gamma}\lbra\rho_{\lambda\gamma}(\!u, w)\rbra_{n_{\beta}}
\lbrac \partial_\alpha z\rbrac
\\
-\frac13\int_{\t\E^0_h}
a^{\alpha\beta\lambda\gamma}\lbra\rho_{\lambda\gamma|\beta}(\!u, w)
\rbra_{n_{\alpha}}\lbrac z\rbrac
+\frac13\int_{\t\E^0_h}a^{\alpha\beta\gamma\delta}\lbra\gamma_{\gamma\delta}(\!u, w)\rbra_{n_{\beta}}\lbrac v_{\alpha}\rbrac
 \\
-\frac23\int_{\t\E^0_h}a^{\gamma\beta\sigma\tau}\lbrac\rho_{\sigma\tau}(\!v, z)\rbrac b^{\alpha}_{\beta}
\lbra u_{\alpha}\rbra_{n_{\gamma}}
-\frac13\int_{\t\E^0_h}
a^{\alpha\beta\lambda\gamma}\lbrac\rho_{\lambda\gamma}(\!v, z)\rbrac
\lbra \partial_\alpha w\rbra_{n_{\beta}}
\\
+
\frac13\int_{\t\E^0_h}
a^{\alpha\beta\lambda\gamma}\lbrac\rho_{\lambda\gamma|\beta}(\!v, z)
\rbrac\lbra w\rbra_{n_{\alpha}}
-\frac13\int_{\t\E^0_h}a^{\alpha\beta\gamma\delta}\lbrac\gamma_{\gamma\delta}(\!v, z)\rbrac\lbra u_{\alpha}\rbra_{n_{\beta}}
\\
+\int_{\t\E^0_h}\lbra\M^{\alpha\beta}\rbra_{n_{\beta}}\lbrac v_{\alpha}\rbrac
+\int_{\t\E^0_h}\lbrac\N^{\alpha\beta}\rbrac\lbra u_{\alpha}\rbra_{n_{\beta}}
\end{multline*}
\begin{multline*}
+\int_{\t\E^F_h}\left[2\frac13a^{\gamma\beta\sigma\tau}\rho_{\sigma\tau}(\!u, w)b^{\alpha}_{\beta}
{n_{\gamma}}
+\frac13a^{\alpha\beta\gamma\delta}\gamma_{\gamma\delta}(\!u, w){n_{\beta}}
+\M^{\alpha\beta}{n_{\beta}}\right]v_{\alpha}
\\
-
\int_{\t\E^F_h}
\frac13\left\{a^{\alpha\beta\lambda\gamma}\rho_{\lambda\gamma|\beta}(\!u, w)
{n_{\alpha}}+ D_{\!s}\left[a^{\alpha\beta\lambda\gamma}\rho_{\lambda\gamma}(\!u, w)
n_{\beta}s_\alpha\right]\right\}z
\\
+\int_{\t\E^F_h\cup\t\E^S_h}
\frac13a^{\alpha\beta\lambda\gamma}\rho_{\lambda\gamma}(\!u, w)
n_{\beta}n_\alpha D_{\!n}z
\end{multline*}
\begin{multline*}
-\int_{\t\E^S_h\cup\t\E^D_h}\left[2\frac13a^{\gamma\beta\sigma\tau}\rho_{\sigma\tau}(\!v, z)b^{\alpha}_{\beta}
{n_{\gamma}}
+\frac13a^{\alpha\beta\gamma\delta}\gamma_{\gamma\delta}(\!v, z){n_{\beta}}
-\N^{\alpha\beta}{n_{\beta}}\right]u_{\alpha}\\
+
\int_{\t\E^S_h\cup\t\E^D_h}
\frac13\left\{a^{\alpha\beta\lambda\gamma}\rho_{\lambda\gamma|\beta}(\!v, z)
{n_{\alpha}}+ D_{\!s}\left[a^{\alpha\beta\lambda\gamma}\rho_{\lambda\gamma}(\!v, z)
n_{\beta}s_\alpha\right]\right\}w\\
-\int_{\t\E^D_h}
\frac13a^{\alpha\beta\lambda\gamma}\rho_{\lambda\gamma}(\!v, z)
n_{\beta}n_\alpha D_{\!n}w.
\end{multline*}

Since there is no jump in the Koiter model solution $\!u\e, w\e, \M\e$, in view of the definition \eqref{form_a} and \eqref{form_ub_a}
we have $a(\!u\e, w\e; \!v, z)=\ub a(\!u\e, w\e; \!v, z)$ for any piecewise $\!v, z$.
We therefore have that for any piecewise regular test function $\!v, z$
\begin{multline*}
a(\!u\e, w\e;\!v, z)+b(\M\e; \!v, z)-b(\N; \!u\e, w\e)+\eps^2c(\M\e, \N)\\
=
\ub a(\!u\e, w\e;\!v, z)+b(\M\e; \!v, z)-b(\N; \!u\e, w\e)+\eps^2c(\M\e, \N)=
\langle\!f; \!v, z\rangle.
\end{multline*}
The linear form $\langle\!f; \!v, z\rangle$ is defined by \eqref{form_f}. The last equation follows from comparing
\eqref{K-ub-a-weak} with the strong form of the Koiter model \eqref{K-P-strong}.
This verifies the consistency of the finite element model
\eqref{K-fem} with the Koiter shell model \eqref{K-P-model}.

\end{document}